\documentclass[11pt,a4paper,twoside]{article}
\usepackage{latexsym,amsfonts,amsmath, amsthm, amssymb}
\usepackage{fullpage}
\usepackage{xcolor,bm}
\usepackage[utf8x]{inputenc}
\usepackage{array}

\usepackage{fourier}

\newtheorem{theorem}{Theorem}
\newtheorem{proposition}[theorem]{Proposition}
\newtheorem{lemma}[theorem]{Lemma}

\theoremstyle{definition}

\newtheorem{remark}[theorem]{Remark}
\newtheorem{conjecture}[theorem]{Conjecture}

\newcommand{\ls}{\leqslant}
\newcommand{\gs}{\geqslant}

\renewcommand{\Re}{\operatorname{Re}}
\renewcommand{\Im}{\operatorname{Im}}

\begin{document}

\title{\bf Selberg-type integrals and the variance conjecture for the operator norm}

\medskip

\author{Beatrice-Helen Vritsiou}

\date{}

\maketitle

\begin{abstract}
\small
%\footnotesize
The variance conjecture in Asymptotic Convex Geometry stipulates
that the Euclidean norm of a random vector uniformly distributed in a (properly normalised) 
high-dimensional convex body $K\subset {\mathbb R}^n$ satisfies a Poincaré-type inequality,
implying that its variance is much smaller than its expectation.
We settle the conjecture for the cases when $K$ is the unit ball of the operator norm 
in classical subspaces of square matrices, which include the subspaces of self-adjoint matrices.
Through the estimates we establish, we are also able to show that the unit ball of the operator norm
in the subspace of real symmetric matrices or in the subspace of Hermitian matrices is
not isotropic, yet is in almost isotropic position.
\end{abstract}

\section{Introduction}\label{sec:intro}

This note is a follow-up on \cite{Radke-V-2016}, in which we were concerned with the question whether 
the variance (or thin-shell) conjecture holds true for unit balls of the $p$-Schatten norms. 
Given a convex body $K$ in ${\mathbb R}^m$, that is, a convex, compact set with non-empty interior, 
whose covariance matrix ${\rm Cov}(K)$, given by
\begin{equation} \label{def:cov-matrix} {\rm Cov}(K)_{i,j} : = 
\frac{\int_K x_ix_j\, dx}{\int_K {\bf 1}\,dx}\ - \ \frac{\int_K x_i\,dx}{\int_K {\bf 1}\,dx}\,\frac{\int_K x_j\,dx}{\int_K {\bf 1}\,dx}
\quad\quad \hbox{for}\ 1\ls i, j\ls m,
\end{equation}
has small condition number, the variance conjecture is a statement that most of the mass of $K$
will be found in an annulus of width much smaller than its average radius, a ``thin shell''
(see the \emph{$\varepsilon$-Concentration Hypothesis} of Anttila, Ball and Perissinaki \cite{Anttila-Ball-Perissinaki-2003},
or the quantitatively stronger statement \eqref{eq:var-conj} below suggested by Bobkov and Koldobsky 
\cite{Bobkov-Koldobsky-2003}).
Supposing first for simplicity 
that $K$ has Lebesgue volume 1, 
barycentre at the origin, 
and that $K$ is \emph{isotropic}, that is, ${\rm Cov}(K)$ is a multiple of the identity matrix, 
the conjecture can be stated as asking that
\begin{equation} \label{eq:var-conj}
{\rm Var}_K\bigl(\|x\|_2^2\bigr): = \int_K\|x\|_2^4\,dx - \left(\int_K\|x\|_2^2\,dx\right)^2 \lesssim \frac{1}{m} \left(\int_K\|x\|_2^2\,dx\right)^2\, ,
\end{equation}
where $\|\cdot\|_2$ stands for the Euclidean norm on ${\mathbb R}^m$, 
and `$\lesssim$' implies a multiplicative constant that should not depend on the dimension $m$ or the body $K$.

Although stated separately and with different motivations initially,
inequality \eqref{eq:var-conj} is a special case  
of the KLS conjecture
(put forth by Kannan, Lovász and Simonovits \cite{KLS-1995})
when the latter is equivalently reformulated as a Poincaré inequality for convex bodies
(the equivalence following by works of Maz'ya, Cheeger, Buser and Ledoux):
according to this, 
given a convex body 
$K\subset {\mathbb R}^m$ 
of volume 1 with barycentre at the origin, and any (locally) Lipschitz function 
$f: {\mathbb R}^m \to {\mathbb R}$, 
we should have
\begin{equation} \label{eq:Poin-conj} 
{\rm Var}_K\bigl(f\bigr)\, \lesssim \,
s_{\rm max}\bigl[{\rm Cov}(K)\bigr]\cdot \int_K\big\|\nabla f(x)\big\|_2^2\,dx\, , 
\end{equation}
where $s_{\rm max}\bigl[{\rm Cov}(K)\bigr]$ denotes the largest singular value of the covariance matrix of $K$. 
To see the connection, observe that when ${\rm Cov}(K)$ is a multiple of the identity matrix,  
we have
\begin{equation}\label{eq:max-avg-s-val} s_{\rm max}\bigl[{\rm Cov}(K)\bigr] = \frac{1}{m}{\rm tr}\bigl[{\rm Cov}(K)\bigr]
= \frac{1}{m} \int_K\|x\|_2^2\,dx.
\end{equation}
Of course, with the KLS conjecture in mind,
it makes sense to ask about the validity of a suitably modified inequality \eqref{eq:var-conj}
even when ${\rm Cov}(K)$ is not a multiple of the identity, 
and when \eqref{eq:max-avg-s-val} is not true even approximately
(or we don't know a priori whether it is).

\begin{conjecture}(``Generalised Variance Conjecture'') \label{conj:gen-var}
There is an absolute constant $C$ such that, 
given any convex body $K\subset {\mathbb R}^m$ of volume 1 with barycentre at the origin, one has
\begin{equation} \label{eq:gen-var-conj} {\rm Var}_K\bigl(\|x\|_2^2\bigr) \ls C\cdot  s_{\rm max}\bigl[{\rm Cov}(K)\bigr] \int_K\|x\|_2^2\,dx.\end{equation}
\begin{remark}
The assumption that $K$ has volume 1 is merely for convenience: if we don't make it, integration above is understood
instead with respect to the density ${\bf 1}_K(x)/ (\int_K {\bf 1}\,dx)$. 
\end{remark}
\end{conjecture}

In this note we verify this conjecture for the unit ball of the operator norm 
on several classical subspaces of square matrices.

Before we turn to particulars, let us recall that, despite the fact
that Conjecture 1, or its more restricted version for isotropic convex bodies only,
seem like very special cases of the KLS conjecture, they are in fact almost
equivalent reformulations of it:
according to a surprising result by Eldan \cite{Eldan-2013}, whatever estimates
one obtains for the constant $C$ appearing in \eqref{eq:gen-var-conj} (\emph{for all} centred convex bodies),
or even just for inequality \eqref{eq:var-conj} (\emph{for all} isotropic convex bodies),
the same estimates (up to a multiplicative logarithmic factor in the dimension $m$) will also be valid for the implied constant in \eqref{eq:Poin-conj}.
Estimates for the constant $C = C(m)$ in \eqref{eq:var-conj} depending on the dimension have been obtained by Klartag \cite{Klartag-2007a},
\cite{Klartag-2007b}, by Fleury, Guédon and Paouris \cite{Fleury-Guedon-Paouris-2007}, Fleury \cite{Fleury-2010b},
and by Guédon and Milman \cite{Guedon-Milman-2011} 
%the latter showing that $C(m) = O\bigl(m^{2/3}\bigr)$
(moreover, prior to Eldan's result, estimates for the implied constant in the Poincaré inequality \eqref{eq:Poin-conj}
had been obtained by Kannan, Lovász and Simonovits \cite{KLS-1995} and by Bobkov \cite{Bobkov-2007}).
A recent improvement to all these is given by Lee and Vempala \cite{Lee-Vempala-2016},
who established inequality \eqref{eq:Poin-conj} with $C(m) = O(\sqrt{m})$. 

As far as specific cases of convex bodies are concerned, inequality \eqref{eq:var-conj}
has been established (optimally) for the unit balls of the $\ell_p$ norms by Ball
and Perissinaki \cite{Ball-Perissinaki-1998}, for isotropic unconditional convex bodies
by Klartag \cite{Klartag-2009}, and, via extending Klartag's method in \cite{Klartag-2009},
 by Barthe and Cordero-Erausquin \cite{Barthe-Cordero-2013} 
for isotropic (or almost isotropic) convex bodies that have many symmetries 
(maybe fewer than those of an unconditional body, but still enough; one such example is the simplex, 
or any other convex body which has the symmetries of the simplex).
Furthermore, Conjecture 1 has been verified by Alonso-Gutiérrez and Bastero \cite{Alonso-Bastero-2016}
for hyperplane projections of the unit balls of the $\ell_p$ norms.

For background on and further results related to these conjectures, we refer the reader to the books
\cite{Alonso-Bastero-book} and \cite{BGVV-book}.

\medskip

We now state the main result of this note.
Let ${\cal M}_n({\mathbb F})$ denote the space 
of all $n\times n$ matrices with entries from the division algebra ${\mathbb F}$,
which stands either for ${\mathbb R}$ or ${\mathbb C}$ or the skew field ${\mathbb H}$ of quaternions 
(note that in all cases we view ${\cal M}_n({\mathbb F})$ as a real vector space,
which can thus be thought of as ${\mathbb R}^m$ where $m = \beta n^2$ with $\beta = 1, 2$ or $4$ respectively). 
For a matrix $T\in {\cal M}_n({\mathbb F})$ and $p\gs 1$, the $p$-Schatten norm of $T$ is given by
\begin{equation*} \|T\|_{S_p^n} := \|s(T)\|_p = \left(\sum_{i=1}^n s_i(T)^p\right)^{1/p},\end{equation*}
where $s(T) = (s_1(T),\ldots, s_n(T))$ is the non-increasing rearrangement of the singular values of $T$, that is, of the eigenvalues of $(T^\ast T)^{1/2}$.
The limiting case of $p = \infty$ is defined in the usual way: $\|T\|_{S_\infty^n} : = \|s(T)\|_\infty = s_{\rm max}(T)$
is the \emph{operator} or \emph{spectral} norm of $T$. 
Also, the Euclidean norm $\|\cdot\|_2$ on ${\cal M}_n({\mathbb F})$ coincides with the 
2-Schatten norm $\|\cdot\|_{S_2^n}$, also known as the \emph{Hilbert-Schmidt} or \emph{Frobenius} norm.

We will focus on establishing Conjecture \ref{conj:gen-var} when $K$ is 
the unit ball of $S_\infty^n$ on either of the spaces ${\cal M}_n({\mathbb F})$, 
or moreover on its classical subspace of ${\mathbb F}$-self-adjoint matrices.

\begin{theorem}\label{thm:var-operator}
Let ${\mathbb F}$ stand for either ${\mathbb R}$ or ${\mathbb C}$ or ${\mathbb H}$,
and let $E = {\cal M}_n({\mathbb F})$ or the subspace of ${\mathbb F}$-self-adjoint matrices.
Set $d_n = {\rm dim}(E)$, and write $B_E$ for the unit ball of $\|\cdot\|_{S_\infty^n}$ on $E$, 
and $\overline{B_E}$ for its homothetic copy of volume $1$, that is, $\overline{B_E} := \frac{B_E}{[{\rm vol}(B_E)]^{1/{d_n}}}$. 
Then there are absolute constants $C_1>0, C_2$ so that
\begin{equation} \label{eq:thm:var-operator}
C_1 \ls \sigma_{B_E}^2 := d_n\,\frac{{\rm Var}_{\overline{B_E}}\bigl(\|T\|_{S_2^n}^2\bigr)}{\left(\displaystyle\int_{\overline{B_E}} \|T\|_{S_2^n}^2\,dT\right)^2} \ls C_2.
\end{equation}
\end{theorem}
\begin{remark}
Obviously this implies Conjecture \ref{conj:gen-var} for the (normalised) unit ball $\overline{B_E}$ of the operator norm on $E$
since we always have $\frac{1}{d_n} \int_{\overline{B_E}} \|T\|_{S_2^n}^2\,dT = \frac{1}{d_n} {\rm tr}\bigl[{\rm Cov}\bigl(\overline{B_E}\bigr)\bigr] \ls 
s_{\rm max}\bigl[{\rm Cov} \bigl(\overline{B_E}\bigr)\bigr] $.
\end{remark}

For most of the cases of $E$ mentioned above these estimates were also established 
in \cite{Radke-V-2016} (with somewhat similar methods as we will see);
however, for the subspaces of symmetric (or real self-adjoint) matrices,
and of quaternionic self-adjoint matrices, the result is new.

The previous best result in the case of $E = {\cal M}_n({\mathbb F})$ 
followed from the method of Barthe and Cordero-Erausquin in \cite{Barthe-Cordero-2013}: 
they showed that the unit ball of the operator norm  
has sufficiently many symmetries for us to conclude
that $\sigma_{B_E}^2 = O(n) = O\bigl(\sqrt{{\rm dim}(E)}\bigr)$;
in fact the same was shown true for the unit balls in ${\cal M}_n({\mathbb F})$ of all the other $p$-Schatten norms. 
This upper bound now also follows from \cite{Lee-Vempala-2016}. 

Note, however, that it has been unclear whether either approach implies the same bound 
in the subspaces of self-adjoint matrices given that it wasn't known (to the best of our knowledge) 
if the condition number of the covariance matrix of $B_E$ in such a subspace $E$ is small 
(similarly this appears not to be known for any other $p$-Schatten norm besides $p=2$).
In this note we show this condition number
to be small, at least when $E$ consists of the real or complex self-adjoint matrices (see Theorem \ref{thm:almost-isotropic} below). 
Observe nevertheless that the estimates in \eqref{eq:thm:var-operator} 
are established regardless of that.

\bigskip

The starting point here, as well as for the arguments in \cite{Radke-V-2016},
is the fact that the uniform distribution on $\overline{B_E}$ defines
an \emph{invariant} ensemble of `random' matrices from $E$:
the distribution remains the same under multiplication by
an ${\mathbb F}$-unitary matrix (by which we understand either multiplication from left or from right
when $E = {\cal M}_n({\mathbb F})$, or conjugation by the matrix when 
$\overline{B_E}$ 
contains only ${\mathbb F}$-self-adjoint matrices).
Equivalently, the distribution depends only 
on the non-increasing rearrangement of the singular values $s_i(T)$ of $T\in E$ when $E = {\cal M}_n({\mathbb F})$, 
or of the eigenvalues $e_i(T)$ of $T\in E$ when $E$ consists of the ${\mathbb F}$-self-adjoint matrices.
As a consequence the integrals in \eqref{eq:thm:var-operator} which we wish to estimate,
given also that the integrands depend only on the singular values of $T$, can be reduced to 
integrals of highly symmetric distributions over ${\mathbb R}^n$ 
(see Lemma \ref{lem:RMT-reduction} and Proposition \ref{prop:var-reduction}).

It follows that to estimate $\sigma_{B_E}^2$,
it is completely equivalent to obtain estimates for the variance of the Euclidean norm with respect to the density
\begin{equation*}  
x=(x_1,\ldots,x_n)\in {\mathbb R}^n \quad \mapsto \quad {\bf 1}_{[-1,1]^n}(x) \prod_{1\ls i<j\ls n}\big|\,x_i^a - x_j^a\big|^b\cdot \prod_{1\ls i\ls n}|x_i|^c\,dx,\end{equation*}
where $a,b,c$ are integers depending only on $E$ ($a\in \{1,2\}$, $b=\beta = {\rm dim}_{{\mathbb R}}({\mathbb F})$, and $c\in \{0,\beta-1\}$).
This requires us to study integrals of the form
\begin{equation} \label{eq:Selberg-type-int-1} 
\int_{-1}^1\int_{-1}^1\cdots \int_{-1}^1 s(x)\,\prod_{1\ls i<j\ls n}\big|\,x_i^a - x_j^a\big|^b\cdot \prod_{1\ls i\ls n}|x_i|^c\,dx_n\ldots dx_2dx_1\end{equation}
where $a=1$ or 2, and where the integrand $s(x)$ is a symmetric polynomial 
(in this case we will have $s(x) = \sum_i x_i^k$ with $k=2$ or 4, or $s(x)=\sum_{i<j}x_i^2 x_j^2$).

\medskip

With suitable changes of variables, all such integrals can be related to integrals of a similar form:
\begin{equation}\label{eq:Selberg-type-int-2} \int_0^1\int_0^1\cdots \int_0^1 \tilde{s}(t) \prod_{1\ls i\ls n} t_i^{u-1} (1-t_i)^{w-1}\,
\prod_{1\ls i<j\ls n}\big|\,t_i - t_j\big|^{2\kappa}\,dt_n\ldots dt_2dt_1\end{equation}
where again $\tilde{s}(t)$ is a symmetric polynomial, and where $u > 0$, $w>0$ and $\kappa \gs 0$
(we can even think of $u, w,\kappa$ as complex numbers, 
with the inequalities-constraints then holding for their real part).
Selberg \cite{Selberg-1944} was the first to study such a family of integrals in the case where $\tilde{s}(t) = {\bf 1}$
(using crucially the fact that the change of variables $t_i \mapsto 1-t_i$ leaves the integrals in this family unchanged),
and he showed that each of them equals a certain product of Gamma factors (that is, of values of the Gamma function)
whose inputs depend only linearly on $u, w$ and $\kappa$ in a pre-specified manner:
\begin{multline} \label{eq:Selberg-formula}
I_0(n; u, w, \kappa) : = \int_0^1\int_0^1\cdots \int_0^1 \prod_{1\ls i\ls n} t_i^{u-1} (1-t_i)^{w-1}\,
\prod_{1\ls i<j\ls n}\big|\,t_i - t_j\big|^{2\kappa}\ dt_n\ldots dt_2dt_1
\\
= \prod_{1\ls i \ls n}\frac{\Gamma\bigl(1+(n-i+1)\kappa\bigr)}{\Gamma(1+\kappa)} 
\ \prod_{1\ls i\ls n}\frac{\Gamma\bigl(u + (n-i)\kappa\bigr)\, \Gamma\bigl(w + (n-i)\kappa\bigr)}{\Gamma\bigl(u+w + (2n-i-1)\kappa\bigr)}.
\end{multline} 

Aomoto \cite{Aomoto-1987}, and then Kadell \cite{Kadell-1997}, 
the latter confirming a conjecture by Macdonald \cite[Conjecture (C5)]{Macdonald-1987}, 
have generalised this result by 
establishing completely analogous `closed-form' expressions for the corresponding integrals 
when $\tilde{s}(t)$ ranges in different families of non-constant symmetric polynomials.
In fact, Kadell's  
result encompasses all the previous results since the family of polynomials $\tilde{s}(t)$
which he proves one can consider contains the family of Jack symmetric polynomials
(under a standard normalisation) and therefore spans the space of symmetric polynomials
(see Subsection \ref{subsec:SAK-results} for definitions and specifics;
also, for other proofs of Kadell's result, see Kaneko \cite{Kaneko-1993},
Baker and Forrester \cite{Baker-Forrester-1998} (see also \cite{Forrester-Warnaar-2008}
for a streamlined sketch of this proof), and Warnaar \cite{Warnaar-2008}).

\smallskip

In Section \ref{sec:main-proof} we show how to use Aomoto's result (as well as an immediate extension of it)
in order to recover the conclusion of Theorem \ref{thm:var-operator} when $E = {\cal M}_n({\mathbb F})$,
and furthermore how to use Kadell's  
more general result to obtain Theorem \ref{thm:var-operator} for the subspaces of self-adjoint matrices too.

\bigskip

The estimates we obtain
for integrals of the form \eqref{eq:Selberg-type-int-1} allow us to also
deal with the question of what the covariance matrix of $\overline{B_E}$ is when $E$ is one of the subspaces of self-adjoint matrices.
Note that in the cases of the spaces ${\cal M}_n({\mathbb F})$ it is not difficult to see 
that simply the symmetries/invariances of the respective unit balls 
$B_{{\cal M}_n({\mathbb F})}$ (and similarly of the unit balls of all other $p$-Schatten norms) guarantee these bodies are isotropic
(see e.g. \cite[Proposition 26]{Radke-V-2016}); however in the case of the subspaces of self-adjoint matrices the symmetries are no longer enough for a similar conclusion.

Let us observe that, since $\overline{B_E}$ has volume 1 and the origin as a centre of symmetry, 
computing the entries of the covariance matrix as in \eqref{def:cov-matrix} reduces essentially to computing integrals of the form
\begin{equation} \label{eq:covariance-matrix-entries}
\int_{\overline{B_E}} |T_{i,j}|^2\,dT,\quad 1\ls i,j\ls n, \qquad\hbox{as well as}\quad  \int_{\overline{B_E}} T_{i,j}T_{l,k}\,dT\quad\hbox{for}\ (i,j)\neq (l,k).
\end{equation}
This is made possible through the Weingarten calculus which allows to estimate integrals of polynomial functions of the entries of a random matrix
belonging to several important types of matrix ensembles by relating them to integrals of symmetric functions of the eigenvalues of these matrices: 
for our setting we need a result of Collins, Matsumoto and Saad \cite{Collins-Matsumoto-Saad-2014} for conjugate invariant ensembles of self-adjoint matrices 
with real or complex entries (see Subsection \ref{subsec:Weingarten} for details). 

The estimates we obtain are summarised in the following theorem, 
and show that $\overline{B_E}$ is almost isotropic when $E$ is the subspace of symmetric matrices, 
or the subspace of Hermitian matrices (see \ref{sec:almost-isotropic} for the details 
and more precise estimates including constants).

\begin{theorem}\label{thm:almost-isotropic}
Let $E$ be the subspace of $\ {\mathbb F}$-self-adjoint matrices with $\,{\mathbb F} = {\mathbb R}$ 
or $\,{\mathbb C}$. Then all integrals of the first form in \eqref{eq:covariance-matrix-entries}
are of the order of 1, while all integrals of the second form are zero except when $i=j\neq l=k$.
In fact, when $\,{\mathbb F} = {\mathbb C}$ and, say, $i\neq j$, we also have
\begin{gather*}
\int_{\overline{B_E}}\Re(T_{i,j})\Re(T_{l,k})\,dT = \int_{\overline{B_E}} \Im(T_{i,j})\Im(T_{l,k})\,dT
= \int_{\overline{B_E}} \Re(T_{i,j})\Im(T_{l,k})\,dT = 0,
\\
\hbox{as well as}\quad \int_{\overline{B_E}} \Re(T_{i,j})\Im(T_{i,j})\,dT = 0.
\end{gather*} 

On the other hand, when $i=j\neq l=k$, we have
\begin{equation*}
\int_{\overline{B_E}} T_{i,i}T_{k,k}\,dT \simeq -\frac{1}{n}.
\end{equation*}
\end{theorem}

\smallskip

The rest of the paper is organised as follows. 
In Section \ref{sec:prelim} we give exact statements for all the abovementioned 
results that we need. 
Theorem \ref{thm:var-operator} and Theorem \ref{thm:almost-isotropic} are proven in Sections \ref{sec:main-proof} and \ref{sec:almost-isotropic} respectively.

We recall finally that Collins, Matsumoto and Saad deal in \cite{Collins-Matsumoto-Saad-2014} also with the case of left-right invariant ensembles
(which covers e.g. integration of polynomial functions over $B_{{\cal M}_n({\mathbb F})}$). 
In Section \ref{sec:neg-cor-prop}
we exploit this to add to and complete the conclusions from \cite{Radke-V-2016} concerning the question whether
the entries of $T\sim {\rm Unif}\bigl(B_{{\cal M}_n({\mathbb F})}\bigr)$ are negatively correlated in a certain sense
(for the precise definitions and statements see Section \ref{sec:neg-cor-prop}).

\section{Preliminaries and overview of key prior results}\label{sec:prelim}

We will denote by $\|\cdot\|_p$ the $\ell_p$ norm on ${\mathbb R}^n$ and by $B_p^n$ its unit ball,
namely $B_p^n = \bigl\{x\in {\mathbb R}^n: \|x\|_p = \sum_{i=1}^n|x_i|^p\ls 1\bigr\}$.

Let $S_n$ be the symmetric group of permutations of the elements
of $[n]:=\{1,2,\ldots,n\}$. We will say a function $F:{\mathbb R}^n \to {\mathbb R}$ is symmetric  
if $F(x_1,x_2,\ldots,x_n) = F(x_{\sigma(1)},x_{\sigma(2)},\ldots,x_{\sigma(n)})$ for every $\sigma \in S_n$.
Given $s\in {\mathbb R}$, we will say $F$ is $s$-homogeneous if, for every $t>0$, we have $F(tx) = t^sF(x)$.

Let $n$ be a positive integer. A \emph{partition} $\lambda$ of $n$ is a sequence of positive integers $(\lambda_1,\ldots,\lambda_m)$
such that $\lambda_1\gs\cdots \gs \lambda_m$ and $\sum_{i=1}^m \lambda_i = n$; in such a case we write $\lambda \vdash n$
or $|\lambda| = n$. 
The integers $\lambda_i$ are called the \emph{parts} of $\lambda$, and their total number is the \emph{length} of $\lambda$ and is denoted by $l(\lambda)$.
Sometimes we may need to consider sequences with a fixed number of terms, say $m_0$,
in which case we will think of all partitions $\lambda$ with $l(\lambda)\ls m_0$ as giving
such sequences once we annex to them a finite number of parts all equal to 0 as necessary
(in this case $l(\lambda)$ will just be the number of non-zero parts, 
and we can also speak of partitions of 0 all of whose parts are necessarily 0).

Given a partition $\lambda$, the \emph{monomial symmetric function} $m_\lambda({\bm t})$
in $n$ variables, where $n\gs l(\lambda)$, is given by 
\begin{equation*} m_\lambda(t_1,\ldots,t_n) = \frac{1}{|{\rm Stab}(\lambda)|}\,
\sum_{\sigma\in S_n}t_{\sigma(1)}^{\lambda_1} \cdots t_{\sigma(1)}^{\lambda_1}, \end{equation*} 
where $|{\rm Stab}(\lambda)|$ denotes the order of the stabiliser of any monomial of type $\lambda$ under the action of $S_n$
(and dividing by it ensures we add each monomial only once).
By convention, $m_\lambda(t_1,\ldots,t_n) = 0$ if $n< l(\lambda)$. Moreover, when $\lambda = (1,1,\ldots,1) = (1^k)$
for some $k\gs 1$, then we may also write $e_k({\bm t})$ instead of $m_{(1^k)}({\bm t})$ 
and call this the \emph{$k$-th elementary symmetric function}.

\smallskip

The letters $c,c^{\prime }, c_1, c_2$ etc. denote absolute positive
constants (which do not depend on the dimension of the Euclidean space we're in, or moreover 
on any of the other parameters unless specifically stated); their value may change from line to line. 
We will use the notation $A\simeq B$ (or $A\lesssim B$) to mean there exist absolute constants
$c_1,c_2>0$ such that $c_1A\leq B\leq c_2A$ (or $A\ls c_1B$). We will also use the Landau notation:
$A = O(B)$ has the same meaning as $A\lesssim B$, whereas $A = o(B)$ will mean 
the ratio $A/B$ tends to 0 as the dimension grows to infinity.

\medskip

Recall that the uniform distribution over the unit ball of any $p$-Schatten norm in ${\cal M}_n({\mathbb F})$
or its subspace of self-adjoint matrices defines an invariant ensemble of random matrices:
we will call this \emph{left-right invariant} ensemble if the distribution remains unchanged under multiplication either from the left or from the right 
by a fixed ${\mathbb F}$-unitary matrix (this is true in the case of ${\cal M}_n({\mathbb F})$),
and we will call it \emph{conjugate invariant} if the distribution remains unchanged under conjugation by an ${\mathbb F}$-unitary matrix
(this is true in the case of ${\mathbb F}$-self-adjoint matrices). 
Equivalently, the underlying distribution of a left-right invariant ensemble depends only on the distribution
of the non-increasing rearrangement of the singular values of the matrices, 
whereas that of a conjugate invariant ensemble depends only on (the non-increasing rearrangement of) the eigenvalues.

\subsection{Reduction to Selberg-type integrals}\label{subsec:prelim}

A consequence of left-right or conjugate invariance is that estimating integrals of functions that would also only depend on
the singular values or eigenvalues of a matrix $T$ in the ensemble, as for example the implied integrals in Theorem \ref{thm:var-operator},
can be reduced to computing integrals of highly symmetric distributions over ${\mathbb R}^n$ 
(for which there may be more, analytic or combinatorial, tools to use).  
In fact, if we consider the same question for unit balls of the other $p$-Schatten norms,
then (given that the integrands we are interested in, namely powers of the Euclidean norm, are also homogeneous functions) we can equivalently try to estimate  
the corresponding integrals with respect to densities of the form $\exp\bigl(-\|T\|_{S_p^n}^p\bigr)dT$.
Proposition \ref{prop:var-reduction} below was proven in \cite{Radke-V-2016} 
based on the following key fact from Random Matrix Theory described above 
(see for example \cite{Mehta-2004} or \cite[Propositions 4.1.3 and 4.1.1]{Anderson-Guionnet-Zeitouni-2010}
for proofs).

\begin{lemma}\label{lem:RMT-reduction}
Let ${\mathbb F} = {\mathbb R}$ or ${\mathbb C}$ or $\,{\mathbb H}$, and let $F:{\mathbb R}^n\to {\mathbb R}$
be a measurable and symmetric function. 
Let us write $K_{p,E}$ for the unit ball of the $p$-Schatten norm on a subspace $E$ of ${\cal M}_n({\mathbb F})$,
$d_n$ for the dimension of $E$, and $f_{a,b,c}$ for the function 
\begin{equation*}x\in {\mathbb R}^n \mapsto \prod_{1\ls i<j\ls n}\big|\,x_i^a - x_j^a\big|^b\cdot \prod_{1\ls i\ls n}|x_i|^c. \end{equation*}
Then:
\\
(I) if $E = {\cal M}_n({\mathbb F})$, there is a constant $c_n$ depending only on $E$, such that
\begin{equation}\int_{K_{p,E}} F\bigl(s_1(T),\cdots,s_n(T)\bigr)\,dT = c_n \int_{B_p^n} F\bigl(|x_1|,\cdots,|x_n|\bigr) \cdot f_{2,\beta,\beta-1}\,dx, \end{equation}
where $\beta = {\rm dim}_{\mathbb R}({\mathbb F})$;
furthermore, if $p < \infty$, and if $F$ is also $s$-homogeneous for some $s > -d_n$, then
\begin{equation} \int_{K_{p,E}} F\bigl(s_1(T),\cdots,s_n(T)\bigr)\,dT = 
\frac{c_n}{\Gamma\!\left(1+\frac{d_n+s}{p}\right)}\int_{{\mathbb R}^n} F\bigl(|x_1|,\cdots,|x_n|\bigr) e^{-\|x\|_p^p}f_{2,\beta,\beta-1}(x)\,dx. \end{equation}
\smallskip \\
(II) if $E$ is the subspace of $\ {\mathbb F}$-self-adjoint matrices, there is a constant $c_n$ depending only on $E$, such that
\begin{equation}\int_{K_{p,E}} F\bigl(e_1(T),\cdots,e_n(T)\bigr)\,dT = c_n \int_{B_p^n} F(x) \cdot f_{1,\beta,0}\,dx; \end{equation}
similarly, if $p < \infty$ and $F$ is $s$-homogeneous for some $s > -d_n$, then
\begin{equation}\int_{K_{p,E}} F\bigl(e_1(T),\cdots,e_n(T)\bigr)\,dT =
\frac{c_n}{\Gamma\!\left(1+\frac{d_n+s}{p}\right)}\int_{{\mathbb R}^n} F(x) e^{-\|x\|_p^p}f_{1,\beta,0}(x)\,dx. \end{equation}
\end{lemma}

\smallskip

Denote by $M_p(f)$ the integral of a function $f: {\mathbb R}^n \to {\mathbb R}$ with respect to the density
$f_{a,b,c}(x)\cdot e^{-\|x\|_p^p}\,dx$, where $a,b,c$ are going to depend appropriately on the subspace $E$ we consider,
and by $N_p(f)$ the corresponding integral with respect to the density $f_{a,b,c}(x)\cdot {\bf 1}_{B_p^n}(x)\,dx$.
The following proposition, following from Lemma \ref{lem:RMT-reduction}, appears in \cite{Radke-V-2016}.
(Note that one of the facts it relies on is that 
\begin{equation*} \frac{N_p\bigl(\|x\|_2^2\bigr)}{N_p(1)} \simeq n^{1-\frac{2}{p}} \simeq d_n\,[{\rm vol}(K_{p,E})]^{2/d_n}
\quad \hbox{and}\quad \frac{M_p\bigl(\|x\|_2^2\bigr)}{M_p(1)} \simeq n^{1+\frac{2}{p}};\end{equation*} 
these estimates follow by the main results of \cite{SaintRaymond-1984} and \cite{Konig-Meyer-Pajor-1998} and by \cite[Proposition 3]{Guedon-Paouris-2007}.)

\begin{proposition}\label{prop:var-reduction}
For every $p\gs 1$, we have 
\begin{equation*}
\sigma_{K_{p,E}}^2 := d_n\,\frac{{\rm Var}_{\overline{K_{p,E}}}\bigl(\|T\|_{S_2^n}^2\bigr)}{\left(\displaystyle\int_{\overline{K_{p,E}}} \|T\|_{S_2^n}^2\,dT\right)^2}
\simeq n^{4/p}\ {\rm Var}_{N_p}\bigl(\|x\|_2^2\bigr) := n^{4/p}\,\frac{N_p\bigl(\|x\|_2^4\bigr)}{N_p(1)} - 
\left(\frac{N_p\bigl(\|x\|_2^2\bigr)}{N_p(1)}\right)^2,
\end{equation*}
while, if $p<\infty$ too, then
\begin{equation*} {\rm Var}_{M_p}\bigl(\|x\|_2^2\bigr) := \frac{M_p\bigl(\|x\|_2^4\bigr)}{M_p(1)} - 
\left(\frac{M_p\bigl(\|x\|_2^2\bigr)}{M_p(1)}\right)^2  \simeq \max\Bigl\{\sigma_{K_{p,E}}^2, \frac{1}{p}\Bigr\}\cdot n^{4/p}.\end{equation*}
\end{proposition}

In the case of $p = \infty$ it follows that, to accurately estimate $\sigma_{K_{\infty,E}}^2 \equiv \sigma_{B_E}^2$,
we should study integrals of the form
\begin{equation*} 
\int_{-1}^1\int_{-1}^1\cdots \int_{-1}^1 s(x)\,\prod_{1\ls i<j\ls n}\big|\,x_i^a - x_j^a\big|^b\cdot \prod_{1\ls i\ls n}|x_i|^c\,dx_n\ldots dx_2dx_1\end{equation*}
where $a=1$ or 2, and where the integrand $s(x)$ is a symmetric polynomial (here of degree at most 4).

\subsection{Selberg's, Aomoto's, and Kadell's results}\label{subsec:SAK-results}

Recall the formula for the value of the Euler beta integral:
\begin{equation*} \int_0^1x^{u-1}(1-x)^{w-1}\,dx = \frac{\Gamma(u)\Gamma(w)}{\Gamma(u+w)},\end{equation*}
where $\Re(u), \Re(w) > 0$.
Selberg \cite{Selberg-1944} (see also \cite[Chapter 17]{Mehta-2004} for a presentation of his original proof) 
discovered a high-dimensional generalisation of this formula:
for every triple of complex numbers $u, w, \kappa$ with
\begin{equation*} \Re(u) > 0,\quad  \Re(w) > 0, \quad 
\Re(\kappa) > -\min\left(\frac{1}{n},\frac{\Re(u)}{n-1},\frac{\Re(w)}{n-1}\right),\end{equation*}
if we set  
\begin{equation*} h({\bm t};u,w,\kappa) := \prod_{1\ls i\ls n} t_i^{u-1} (1-t_i)^{w-1}\cdot
\prod_{1\ls i<j\ls n}\big|\,t_i - t_j\big|^{2\kappa} \end{equation*}
we have
\begin{equation*} I_0(n; u, w, \kappa) : = \int\limits_{[0,1]^n} h({\bm t};u,w,\kappa)\,d{\bm t} \ \ =  
\prod_{1\ls i \ls n}\frac{\Gamma\bigl(1+(n-i+1)\kappa\bigr)}{\Gamma(1+\kappa)} 
\ \prod_{1\ls i\ls n}\frac{\Gamma\bigl(u + (n-i)\kappa\bigr)\, \Gamma\bigl(w + (n-i)\kappa\bigr)}{\Gamma\bigl(u+w + (2n-i-1)\kappa\bigr)}.\end{equation*}

\bigskip

Aomoto \cite{Aomoto-1987} extended Selberg's result to more general integrals, where the integrand 
could be $h({\bm t};u,w,\kappa)$ multiplied by an elementary symmetric function $e_m({\bm t})$:   
\begin{equation*}e_m({\bm t}) \ := \sum_{1\ls i_1<\cdots < i_m\ls n}t_{i_1}\cdots t_{i_m}
\qquad \hbox{with}\ \ 1\ls m < n. 
\end{equation*}
We observe that by symmetry we have
\begin{align}
\int\limits_{[0,1]^n} e_m({\bm t}) \cdot h({\bm t};u,w,\kappa)\,d{\bm t}
\label{eq1:Aomoto-integrals}
&\ \ = {{n}\choose{m}} \int\limits_{[0,1]^n} \prod_{1\ls i\ls m} t_i \cdot h({\bm t};u,w,\kappa)\,d{\bm t}
\\ \nonumber
\hbox{which Aomoto showed}\quad \quad \qquad \qquad
&\ \ = {{n}\choose{m}} \prod_{i=1}^m \frac{u + (n-i)\kappa}{u+w + (2n-i-1)\kappa} I_0(n; u, w, \kappa)
\end{align}
(recall that $I_0(n; u, w, \kappa)$ is Selberg's integral,
and we can naturally extend this notation by writing $I_m = I_m(n; u, w, \kappa)$ 
for the integral in \eqref{eq1:Aomoto-integrals}). 
In fact, Aomoto used these expressions to conclude that the ratio:
\begin{equation*} \frac{1}{I_0(n; u, w, \kappa)}\,\int\limits_{[0,1]^n} \prod_{1\ls i\ls n}(t_i-y) \cdot h({\bm t};u,w,\kappa)\,d{\bm t} \end{equation*}
is equal to a certain Jacobi polynomial:
\begin{equation*}
\frac{1}{I_0(n; u, w, \kappa)}\,\int\limits_{[0,1]^n} \prod_{1\ls i\ls n}(t_i-y) \cdot h({\bm t};u,w,\kappa)\,d{\bm t}  = 
\frac{n!}{\prod_i (\alpha + \beta + n + i)} \, P_n^{(\alpha, \beta)}(1-2y),
\end{equation*}
where $\alpha = -1 + 2u/\kappa$, $\ \beta = -1 + 2w/\kappa$ and $P_n^{(\alpha, \beta)}$ is the Jacobi polynomial
of degree $n$.

Aomoto's approach relied on finding recurrence relations between the different $I_m$
which would follow from integration by parts.
It should be mentioned  
that our main argument in \cite{Radke-V-2016} was along very similar lines.

\medskip

With only a little more effort (see \cite[Chapter 8]{Andrews-Askey-Roy-1993}), Aomoto's proof method can also give similar formulas when the integrand involves slightly more general symmetric polynomials having terms of the form
\begin{equation*} \prod_{i=1}^{m_1} t_i\ \cdot \prod_{j = m_1 + 1-m_3}^{m_1+m_2 - m_3} (1-t_j) \end{equation*}
where $m_1, m_2, m_3\gs 0$ and $m_3\ls m_1$, $m_1+m_2 -m_3 \ls n$: we have 
\begin{multline}\label{eq3:Aomoto-integrals}
I_{m_1,m_2,m_3} := \int\limits_{[0,1]^n} \prod_{i=1}^{m_1} t_i\ \cdot \prod_{j = m_1 + 1-m_3}^{m_1+m_2 - m_3} (1-t_j) \cdot  h({\bm t};u,w,\kappa)\,d{\bm t}
\\
= \prod_{i=1}^{m_3}\frac{\bigl(u+w+(n-i-1)\kappa\bigr)}{\bigl(u+w + 1 +(2n-i-1)\kappa\bigr)}
  \cdot \frac{\prod\limits_{i=1}^{m_1}\bigl(u+(n-i)\kappa\bigr)\prod\limits_{i=1}^{m_2}\bigl(w+(n-i)\kappa\bigr)}{\prod\limits_{i=1}^{m_1+m_2}\bigl(u+w+(2n-i-1)\kappa\bigr)}\, I_0(n; u, w, \kappa).
\end{multline}
Note that if $m_3>0$, then there is some overlap in factors of the two products, 
something which allows us to get additional factors of the form $t_i(1-t_i)$ for some $i$ only 
(and will allow us, for instance, to exactly compute $\int_{B_E} \|T\|_{S_2^n}^2\,dT$ or $\int_{B_E} \|T\|_{S_4^n}^4\,dT$
when $E = {\cal M}_n({\mathbb F})$).

\bigskip

Kadell \cite{Kadell-1997} (see also Kaneko \cite{Kaneko-1993}, as well as later proofs in \cite{Baker-Forrester-1998}
and \cite{Warnaar-2008}) has extended these results in the most general way: he has shown that,
for each $\kappa \gs 0$, there is an infinite family of homogeneous symmetric polynomials 
$\{s_\lambda^\kappa({\bm t})\}$ indexed by the partitions, which spans the space of symmetric polynomials,
and such that the polynomial corresponding to the partition $\lambda$ has the following properties:
\begin{itemize}
\item $\displaystyle s_\lambda^\kappa(t_1,\ldots, t_n) = m_\lambda({\bm t}) \  + 
\sum\limits_{\substack{\mu\neq \lambda\\ |\mu| = |\lambda|}} a^\kappa_{\lambda, \mu,n}\, m_\mu({\bm t})$
where $n\gs l(\lambda)$, and where $a^\kappa_{\lambda, \mu,n}$ are coefficients which depend on $\kappa$, $\lambda$
and $\mu$, and which might also depend on the number of variables $n$ (but, as we will shortly see, don't).
\item 
For every $n\gs l(\lambda)$ we have $s_\lambda^\kappa(1^n) = \dfrac{f_n^\kappa[\lambda]}{f_n^\kappa[(0)]}$
where 
\begin{equation*} f_n^\kappa[\lambda] := 
\!\!\prod_{\substack{i<j\\ \lambda_i-\lambda_j > 0}}\bigl(\lambda_i-\lambda_j + (j-i)\kappa\bigr)_\kappa
\ \cdot \!\!\prod_{\substack{i<j\\ \lambda_i-\lambda_j = 0}}\frac{j-i}{j-i+1}\cdot \bigl(1 + (j-i)\kappa\bigr)_\kappa
\end{equation*}
and where $(x)_m := \frac{\Gamma(x+m)}{\Gamma(x)}$ stands for the \emph{Pochhammer function} or \emph{rising factorial}
(here $m$ can take non-integer values too), and moreover we have
\begin{align}
\nonumber
 \int\limits_{[0,1]^n} &s_\lambda^\kappa({\bm t})\cdot \prod_{1\ls i\ls n} t_i^{u-1} (1-t_i)^{w-1}\
\prod_{1\ls i<j\ls n}\big|\,t_i - t_j\big|^{2\kappa}\,d{\bm t} \equiv \int\limits_{[0,1]^n} s_\lambda^\kappa({\bm t}) \cdot h({\bm t};u,w,\kappa)\,d{\bm t}
\\ \label{eq:Kadell-integration-formula}
&\qquad \qquad  = I_0(n;u,w,\kappa)
\cdot s_\lambda^\kappa(1^n)\  
\prod_{i=1}^n \frac{\bigl(u + (n-i)\kappa\bigr)_{\lambda_i}}{\bigl(u+w + (2n-i-1)\kappa\bigr)_{\lambda_i}}
\\ \nonumber
& \qquad \qquad = n! \,f_n^\kappa[\lambda]
\ \prod_{1\ls i\ls n}\frac{\Gamma\bigl(u + (n-i)\kappa + \lambda_i\bigr)\, \Gamma\bigl(w + (n-i)\kappa\bigr)}{\Gamma\bigl(u+w + (2n-i-1)\kappa + \lambda_i\bigr)}.
\end{align}
\end{itemize}
This family can in fact be taken to be the family of (monic) Jack polynomials corresponding to the parameter $1/\kappa$,
that is, $s_\lambda^k({\bm t}) = P_\lambda({\bm t};1/\kappa)$ for every partition $\lambda$.

\medskip

Although we will not need this in the sequel, let us recall for the sake of completeness that
one way of defining the family of Jack polynomials $\bigl\{P_\lambda({\bm t};\xi)\bigr\}$ corresponding 
to a parameter $\xi$ is as follows (see e.g. \cite[Chapter VI]{Macdonald-book}).
Recall that, for any non-negative integer $b$, we can define the power-sum function
$p_b(t_1,\ldots,t_n) := \sum_{i=1}^n t_i^b$; we then extend this notion by defining for every partition 
$\lambda = (\lambda_1,\ldots,\lambda_m)$ a power-sum function 
$p_\lambda({\bm t}) := \prod_{j=1}^m p_{\lambda_j}({\bm t})$. 
We can also define a (partial) ordering of the partitions, called the \emph{dominance} ordering, 
by setting $\mu \preccurlyeq \lambda$ if and only if $|\mu| = |\lambda|$ and
$\mu_1+\cdots+\mu_i \ls \lambda_1 +\cdots + \lambda_i$ for every $i\gs 1$.
Finally, consider the field ${\mathbb Q}(\xi)$ of all rational functions of $\xi$ (seen 
as an indeterminate) with coefficients in ${\mathbb Q}$ and also the vector space
${\mathbb Q}(\xi)\bigl[\bigl\{m_\lambda(t_1,\ldots,t_n): \lambda \ \hbox{partition}, l(\lambda)\ls n\bigr\}\bigr]$
of all symmetric polynomials in $n$ variables with coefficients from ${\mathbb Q}(\xi)$.
We can define a scalar product $\langle\cdot,\cdot\rangle_\xi$ on this vector space
by setting
\begin{equation}
\label{def:power-sum-inner-product}
 \left\langle p_\lambda,p_\mu\right\rangle_\xi : = z_\lambda \xi^{l(\lambda)} \cdot {\bm 1}_{\lambda = \mu}, \end{equation}
where $z_\lambda = \prod_{i=1}^{l(\lambda)}a_i!\cdot i^{a_i}$ with $a_i$ being the number of parts of $\lambda$ equal to $i$.
Then the family of Jack polynomials $\{P_\lambda({\bm t};\xi):\lambda\ \hbox{partition}\}$ in $n$ variables is the unique family of functions in ${\mathbb Q}(\xi)\bigl[\{m_\lambda({\bm t})\}\bigr]$ satisfying the following two properties:
\begin{itemize}
\item
\underline{Orthogonality} $\quad \left\langle P_\lambda({\bm t};\xi),P_\mu({\bm t};\xi)\right\rangle_\xi = 0 \ \  $ if $\ \mu \neq \lambda$.
\item
\underline{Triangularity} 
$\quad$ If we write
\begin{equation*} P_\lambda({\bm t};\xi) = \sum\limits_{\mu: l(\mu)\ls n} c(\lambda, \mu,n;\xi)\, m_\mu({\bm t}) \end{equation*} 
for some coefficients $c(\lambda,\mu,n;\xi)\in {\mathbb Q}(\xi)$, 
then $c(\lambda,\mu,n;\xi) \neq 0$ only if $\mu \preccurlyeq \lambda$ and $c(\lambda,\lambda,n;\xi) = 1$.
\end{itemize}
Actually this definition overdetermines the family of Jack polynomials, which means that
a priori it is not clear that there exists any family from ${\mathbb Q}(\xi)\bigl[\{m_\lambda({\bm t})\}\bigr]$ 
which has these two properties. However it can be shown that such a family exists, and then necessarily it is unique.

Moreover, it can be shown that the coefficients $c(\lambda,\mu,n;\xi)$ do not depend on $n$,
and therefore the Jack polynomials have the following stability property:
for every $n_1\gs n_2\gs l(\lambda)$, 
\begin{equation*} 
P_\lambda\bigl((t_1,\ldots,t_{n_2},{\bm 0}_{n_1-n_2});\xi\bigr)
\equiv P_\lambda\bigl((t_1,\ldots,t_{n_2});\xi\bigr).\end{equation*}
For convenience we also set $P_\lambda\bigl((t_1,\ldots,t_m);\xi\bigr) \equiv 0$ if $m<l(\lambda)$.

\medskip

Alternatively, we can obtain the Jack polynomials corresponding to $\xi$ by considering
the eigenfunctions of the following  
operator  
arising in the Calogero-Sutherland model, 
which aims to describe a system of $n$ identical quantum particles on a circle
(see e.g. \cite{Sutherland-1972}, \cite{Sogo-1994}):
\begin{equation*}  D^\ast_\xi = \sum_{i=1}^nx_i\frac{\partial}{\partial x_i}\left(x_i\frac{\partial}{\partial x_i}\right) 
+ \frac{1}{\xi}\sum_{i<j}\frac{x_i + x_j}{x_i-x_j}\left(x_i\frac{\partial}{\partial x_i} - x_j\frac{\partial}{\partial x_j}\right). \end{equation*}
The Jack polynomial $P_\lambda(t_1,\ldots, t_n;\xi)$ is the unique homogeneous and symmetric polynomial
eigenfunction with eigenvalue $\sum_{i=1}^n\bigl(\lambda_i^2 + \frac{1}{\xi}(n-1-2i)\lambda_i\bigr)$
which is monic and whose leading terms are of type $\lambda$ (in other words, we choose the normalisation
$P_\lambda({\bm t};\xi) = m_\lambda({\bm t}) + \sum_{\mu \prec \lambda} c(\lambda, \mu;\xi)m_\mu({\bm t})$).

\smallskip

Setting $\xi$ equal to different non-zero real values (although it has to be noted that the orthogonalising inner
product defined above will be positive definite only for positive real values), we obtain different families of
symmetric polynomials. With $\xi =1$ the corresponding family is the Schur polynomials $\bigl\{P_\lambda({\bm t};1)\bigr\}$,
which are intimately connected with the representation theory of the symmetric groups $S_n$ 
and of the (complex) general linear groups. 
Other important values, and essentially the only ones we care about for the main applications in this paper,
are $\xi=2$, which gives the zonal polynomials $\bigl\{P_\lambda({\bm t};2)\bigr\}$ 
associated with real symmetric matrices,
and $\xi = \frac{1}{2}$, which gives the quaternion zonal polynomials 
$\bigl\{P_\lambda(\cdot;1/2)\bigr\}$
associated with the quaternionic self-adjoint matrices.

\smallskip

What is important to us in this note is having transition matrices from the basis 
$\{s_\lambda^\kappa({\bm t})\}=\{P_\lambda^{1/\kappa}({\bm t})\}$
to the basis of monomial functions of degree up to 4 and vice versa. 
These can be found via the determinantal expressions for the Jack polynomials in terms of the monomial functions
which were established by Lapointe, Lascoux and Morse \cite{Lapointe-Lascoux-Morse-2000}.
They are given in the following tables (and of course, in the specific cases of the special families 
of the Schur or zonal polynomials ($\kappa = 1, 1/2$ or $2$), such tables were known even before
\cite{Lapointe-Lascoux-Morse-2000}).

\begin{center} $P_{(1)}^{1/\kappa} = m_{(1)}$ 
\hspace{2cm}
\begin{tabular}{|c|| c | c |} 
 \hline
  & $m_{(2)}$ & $m_{(1^2)}$ \\ [0.5ex] 
 \hline\hline \\ [-1.5ex]
 $P_{(2)}^{1/\kappa}$ & 1 & $\frac{2\kappa}{\kappa+1}$ \\ [0.7ex] 
 \hline \\ [-1.5ex]
  $P_{(1^2)}^{1/\kappa}$ & 0 & 1 \\  [0.7ex] 
 \hline
\end{tabular} 
\hspace{2cm}
\begin{tabular}{|c|| c | c | c |} 
 \hline
  & $m_{(3)}$ & $m_{(2,1)}$ & $m_{(1^3)}$ \\ [0.5ex] 
 \hline\hline \\ [-1.5ex]
 $P_{(3)}^{1/\kappa}$ & 1 & $\frac{3\kappa}{\kappa+2}$ & $\frac{6\kappa^2}{(\kappa + 1)(\kappa + 2)}$ \\ [0.7ex]
 \hline \\ [-1.5ex]
 $P_{(2,1)}^{1/\kappa}$ & 0 & 1 & $\frac{6\kappa}{2\kappa+1}$ \\ [0.7ex] 
 \hline \\ [-1.5ex]
  $P_{(1^3)}^{1/\kappa}$ & 0 & 0 & 1 \\  [0.7ex] 
 \hline
\end{tabular} 
\end{center}

\begin{center}
\begin{tabular}{|c|| c | c | c | c| c|} 
 \hline
  & $m_{(4)}$ & $m_{(3,1)}$ & $m_{(2^2)}$ & $m_{(2,1^2)}$ & $m_{(1^4)}$ \\ [0.5ex] 
 \hline\hline \\ [-1.5ex]
 $P_{(4)}^{1/\kappa}$ & 1 & $\frac{4\kappa}{\kappa+3}$ & $\frac{6\kappa(\kappa+1)}{(\kappa + 2)(\kappa + 3)}$ 
 & $\frac{12\kappa^2}{(\kappa + 2)(\kappa + 3)}$ & $\frac{24\kappa^3}{(\kappa + 1)(\kappa + 2)(\kappa + 3)}$ \\ [0.7ex]
 \hline \\ [-1.5ex]
 $P_{(3,1)}^{1/\kappa}$ & 0 & 1 & $\frac{2\kappa}{\kappa+1}$ & $\frac{(5\kappa + 3)\kappa}{(\kappa + 1)^2}$ & $\frac{12\kappa^2}{(\kappa + 1)^2}$ \\ [0.7ex] 
 \hline \\ [-1.5ex]
  $P_{(2^2)}^{1/\kappa}$ & 0 & 0 & 1 & $\frac{2\kappa}{\kappa+1}$ & $\frac{12\kappa^2}{(\kappa + 1)(2\kappa + 1)}$ \\  [0.7ex] 
  \hline \\ [-1.5ex]
  $P_{(2,1^2)}^{1/\kappa}$ & 0 & 0 & 0 & 1 & $\frac{12\kappa}{3\kappa+1}$  \\ [0.7ex]
 \hline \\ [-1.5ex]
 $P_{(1^4)}^{1/\kappa}$ & 0 & 0 & 0 & 0 & 1  \\ [0.7ex]
 \hline
\end{tabular} 
\end{center}

\begin{equation} \label{conversion-table1}
\hbox{
\begin{tabular}{|c|| c | c |} 
 \hline \\ [-1.8ex]
  &  $P_{(2)}^{1/\kappa}$ &  $P_{(1^2)}^{1/\kappa}$ \\ [0.7ex] 
 \hline\hline \\ [-1.5ex]
 $m_{(2)}$ & 1 & $-\frac{2\kappa}{\kappa+1}$ \\ [0.7ex] 
 \hline \\ [-1.5ex]
 $m_{(1^2)}$  & 0 & 1 \\  [0.7ex] 
 \hline
\end{tabular} 
\hspace{2cm}
\begin{tabular}{|c|| c | c | c |} 
 \hline \\ [-1.8ex]
  & $P_{(3)}^{1/\kappa}$ & $P_{(2,1)}^{1/\kappa}$ &  $P_{(1^3)}^{1/\kappa}$ \\ [0.7ex] 
 \hline\hline \\ [-1.5ex]
 $m_{(3)}$ & 1 & $-\frac{3\kappa}{\kappa+2}$ & $\frac{6\kappa^2}{(\kappa + 1)(2\kappa + 1)}$ \\ [0.7ex]
 \hline \\ [-1.5ex]
  $m_{(2,1)}$ & 0 & 1 & $-\frac{6\kappa}{2\kappa+1}$ \\ [0.7ex] 
 \hline \\ [-1.5ex]
 $m_{(1^3)}$  & 0 & 0 & 1 \\  [0.7ex] 
 \hline
\end{tabular} 
}
\end{equation}

\begin{equation} \label{conversion-table2}
\hbox{
\begin{tabular}{|c|| c | c | c | c| c|} 
 \hline \\ [-1.8ex]
  & $P_{(4)}^{1/\kappa}$ & $P_{(3,1)}^{1/\kappa}$ & $P_{(2^2)}^{1/\kappa}$ & $P_{(2,1^2)}^{1/\kappa}$ & $P_{(1^4)}^{1/\kappa}$ \\ [0.7ex] 
 \hline\hline \\ [-1.5ex]
 $m_{(4)}$ & 1 & $-\frac{4\kappa}{\kappa+3}$ & $\frac{2\kappa(\kappa-1)}{(\kappa + 1)(\kappa + 2)}$
 & $\frac{4\kappa^2}{(\kappa + 1)^2}$ & $-\frac{24\kappa^3}{(\kappa + 1)(2\kappa + 1)(3\kappa + 1)}$ \\ [0.7ex]
 \hline \\ [-1.5ex]
 $m_{(3,1)}$ & 0 & 1 & $-\frac{2\kappa}{\kappa+1}$ & $-\frac{\kappa(\kappa + 3)}{(\kappa + 1)^2}$ & $\frac{24\kappa^2}{(2\kappa+1)(3\kappa + 1)}$ \\ [0.7ex] 
 \hline \\ [-1.5ex]
  $m_{(2^2)}$ & 0 & 0 & 1 & $-\frac{2\kappa}{\kappa+1}$ & $\frac{12\kappa^2}{(2\kappa + 1)(3\kappa + 1)}$ \\  [0.7ex] 
  \hline \\ [-1.5ex]
  $m_{(2,1^2)}$ & 0 & 0 & 0 & 1 & $-\frac{12\kappa}{3\kappa+1}$  \\ [0.7ex]
 \hline \\ [-1.5ex]
 $m_{(1^4)}$ & 0 & 0 & 0 & 0 & 1  \\ [0.7ex]
 \hline
\end{tabular} 
}
\end{equation}

\subsection{Weingarten calculus for invariant ensembles}\label{subsec:Weingarten}

A permutation $\sigma \in S_k$ can be decomposed into cycles. If the numbers of lengths of cycles are
$\mu_1\gs \mu_2\gs\cdots \gs \mu_l$, then the sequence $\mu=(\mu_1,\mu_2,\ldots,\mu_l)$ is a partition of $k$. We 
will refer to $\mu$ as the cycle-type of $\sigma$. Recall that the different cycle-types correspond to the different
conjugacy classes of $S_k$. Recall also that characters of $S_k$ are class functions, that is, they 
take the same value at permutations belonging to the same conjugacy class or, in other words,
having the same cycle-type.

For the (pairwise non-isomorphic) irreducible representations of $S_k$, there is a
canonical way of identifying each one of them with a unique partition of $k$ and vice-versa
(see e.g. \cite[Section 2.3]{Sagan-book} or \cite[Chapter 4]{Fulton-Harris-book}).
This also gives a natural one-to-one and onto correspondence between the irreducible characters of $S_k$ and partitions of $k$,
which allows us to write the character table of $S_k$ in terms of partitions 
(in fact, to find $\chi^\lambda(\mu)$,
the value of the character correspoding to $\lambda$ at a permutation with cycle-type $\mu$, one can use the Frobenius formula, see e.g. \cite[Proposition 4.37]{Fulton-Harris-book}). 
In our computations in Sections \ref{sec:almost-isotropic} and \ref{sec:neg-cor-prop} we will need to plug in
values of characters of $S_2, S_3$ and $S_4$, so the character tables for these are recalled here:

\begin{equation}
\hbox{
\begin{tabular}{|c|| c | c ||} 
 \hline \\ [-1.8ex]
  &  $(1^2)$ &  $(2)$ \\ [0.7ex] 
 \hline\hline \\ [-1.5ex]
 $\chi^{(2)}$ & 1 & $1$ \\ [0.7ex] 
 \hline \\ [-1.5ex]
 $\chi^{(1^2)}$  & 1 & -1 \\  [0.7ex] 
 \hline
\end{tabular} 
\hspace{2cm}
\begin{tabular}{|c|| c | c | c ||} 
 \hline \\ [-1.8ex]
  & $(1^3)$ & $(2,1)$ &  $(3)$ \\ [0.7ex] 
 \hline\hline \\ [-1.5ex]
 $\chi^{(3)}$ & $1$ & $1$ & $1$ \\ [0.7ex]
 \hline \\ [-1.5ex]
  $\chi^{(2,1)}$ & $2$ & $0$ & $-1$ \\ [0.7ex] 
 \hline \\ [-1.5ex]
 $\chi^{(1^3)}$  & $1$ & $-1$ & $1$ \\  [0.7ex] 
 \hline
\end{tabular} 
}
\end{equation}

\begin{equation}
\hbox{
\begin{tabular}{|c|| c | c | c | c | c ||} 
 \hline \\ [-1.8ex]
  & $(1^4)$ & $(2,1^2)$ & $(2^2)$ & $(3,1)$ & $(4)$ \\ [0.7ex] 
 \hline\hline \\ [-1.5ex]
 $\chi^{(4)}$ & $1$ & $1$ & $1$ & $1$ & $1$ \\ [0.7ex]
 \hline \\ [-1.5ex]
 $\chi^{(3,1)}$ & $3$ & $1$ & $-1$ & $0$ & $-1$ \\ [0.7ex] 
 \hline \\ [-1.5ex]
  $\chi^{(2^2)}$ & $2$ & $0$ & $2$ & $-1$ & $0$ \\  [0.7ex] 
  \hline \\ [-1.5ex]
  $\chi^{(2,1^2)}$ & $3$ & $-1$ & $-1$ & $0$ & $1$  \\ [0.7ex]
 \hline \\ [-1.5ex]
 $\chi^{(1^4)}$ & $1$ & $-1$ & $1$ & $1$ & $-1$  \\ [0.7ex]
 \hline
\end{tabular} 
}
\end{equation}

\subsubsection{The unitary case}

For two sequences ${\bm i} = (i_1,\ldots,i_k)$ and ${\bm i}^\prime = (i_1^\prime,\ldots,i_k^\prime)$
of positive integers and for a permutation $\pi\in S_k$, set
\begin{equation} \delta_\pi({\bm i},{\bm i}^\prime) = \prod_{s=1}^n \delta_{i_{\pi(s)},i_s^\prime},\end{equation}
where $\delta_{i,j} = {\bf 1}_{\{i=j\}}$.

Given a square matrix $A$ and a permutation $\pi\in S_k$ of cycle-type $\mu=(\mu_1,\mu_2,\ldots,\mu_l)$, set
\begin{equation} {\rm Tr}_\pi(A) = \prod_{j=1}^l {\rm Tr}(A^{\mu_j}).\end{equation} 
Finally, given a partition $\lambda$ of $k$ and a number $z\in {\mathbb C}$, define
\begin{equation} C_\lambda(z) = \prod_{i=1}^{l(\lambda)}\prod_{j=1}^{\lambda_i}(z+j-i) \end{equation}
(in the applications below we are going to evaluate $C_\lambda(z)$ at $z=n$; in this case,
this is just the value  at $1^n = (1,\ldots,1)$ of the Jack polynomial $J_\lambda^1({\bm t})
\equiv c_\lambda \cdot P_\lambda^1({\bm t})$ under a different normalisation, see e.g. \cite[Theorem 5.4]{Stanley-1989}).

One of the equivalent ways of defining the \emph{unitary Weingarten function} on $S_k$
with one complex parameter $z\in {\mathbb C}$ 
(see \cite{Collins-Sniady-2006} or \cite{Collins-Matsumoto-Saad-2014})
is the following: 
it is the complex-valued function on $S_k$ given by
\begin{equation}\pi\in S_k \quad \mapsto \quad 
{\rm Wg}^U(\pi;z) := \frac{1}{k!}\sum_{\substack{\lambda\vdash k\\C_\lambda(z)\neq 0}}\frac{\chi^\lambda(e)}{C_\lambda(z)}\,\chi^\lambda(\pi), 
\end{equation} 
where $e$ is the identity permutation in $S_k$. Note that, unless $z\in \{0,\pm 1,\ldots,\pm (k-1)\}$, 
$C_\lambda(z)\neq 0$ for all partitions $\lambda\vdash k$. 
Note also that ${\rm Wg}^U(\pi;z)$ depends only on the cycle-type of $\pi$.

It is convenient to also consider the convolution of two Weingarten functions. Recall that, 
for two complex-valued functions $f_1, f_2$ on $S_k$, 
\begin{equation*}
(f_1\ast f_2)(\pi) := \sum_{\tau \in S_k} f_1(\pi \tau) f_2(\tau^{-1}) = \sum_{\tau \in S_k} f_1(\tau) f_2(\tau^{-1}\pi).
\end{equation*}
We set
\begin{equation}
\pi\in S_k \quad \mapsto \quad 
{\rm Wg}^U(\pi;z,w) := \left({\rm Wg}^U(\cdot;z) \ast {\rm Wg}^U(\cdot;w)\right)(\pi),
\end{equation}
where $z,w\in {\mathbb C}$.

By Schur's lemma and the orthogonality relations it entails
(see also \cite[Theorem 2.13]{Isaacs-book} for a different derivation), 
we can also write

\begin{equation} 
{\rm Wg}^U(\pi;z,w) = \frac{1}{k!}\sum_{\substack{\lambda\vdash k\\C_\lambda(z)C_\lambda(w)\neq 0}}\frac{\chi^\lambda(e)}{C_\lambda(z)C_\lambda(w)}\,\chi^\lambda(\pi).
\end{equation}

\begin{theorem}{\rm (Conjugacy invariance, \cite[Theorem 3.1]{Collins-Matsumoto-Saad-2014})} 
\label{thm:CMS-C-conj-inv}
Let $T = (T_{ij})$ be an $n\times n$ Hermitian random matrix whose distribution
has the property that $UTU^\ast$ is distributed in the same way as $T$ for any unitary matrix $U$. 
For two sequences ${\bm i} = (i_1,\ldots,i_k)$ and ${\bm j} = (j_1,\ldots,j_k)$, we have
\begin{equation*} {\mathbb E}[T_{i_1j_1}T_{i_2j_2}\cdots T_{i_kj_k}] = 
\sum_{\sigma,\tau\in S_k} \delta_\sigma({\bm i},{\bm j})\,{\rm Wg}^U(\sigma^{-1}\tau;n)\ {\mathbb E}[{\rm Tr}_\tau(T)].
\end{equation*}
\end{theorem}

\begin{theorem}{\rm (Left-right invariance, \cite[Theorem 3.4]{Collins-Matsumoto-Saad-2014})} 
\label{thm:CMS-C-lr-inv}
Let $X$ be a complex $n\times p$ random matrix which has the same distribution as $UXV$
for any unitary matrices $U,V$. Then, for four sequences ${\bm i} = (i_1,\ldots,i_k)$, ${\bm i}^\prime = (i_1^\prime,\ldots,i_k^\prime)$,
${\bm j} = (j_1,\ldots,j_k)$ and ${\bm j}^\prime = (j_1^\prime,\ldots,j_k^\prime)$, we have
\begin{equation*} {\mathbb E}\bigl[X_{i_1j_1}X_{i_2j_2}\cdots X_{i_kj_k}\overline{X_{i_1^\prime j_1^\prime}X_{i_2^\prime j_2^\prime}\cdots X_{i_k^\prime j_k^\prime}}\bigr] = 
\sum_{\sigma_1,\sigma_2,\tau\in S_k} \delta_{\sigma_1}({\bm i},{\bm i}^\prime)\delta_{\sigma_2}({\bm j},{\bm j}^\prime)\,
{\rm Wg}^U(\tau\sigma_1^{-1}\sigma_2;n,p)\ {\mathbb E}[{\rm Tr}_\tau(XX^\ast)].
\end{equation*}
\end{theorem}

\begin{remark}
The proof of either theorem proceeds along very similar lines: one notes that $T$ or $X$ has the same 
distribution as $UDU^\ast$ or $UDV^\ast$ respectively, where $D$ is a diagonal matrix 
(with the same distribution of eigenvalues or singular values as $T$ or $X$ respectively),
$U, V$ are Haar-distributed random unitary matrices, and $D$, $U$ and $V$ are all independent.
Then, once the integrals we are interested in are rewritten using these decompositions,
one invokes the following pivotal result in Weingarten calculus (see e.g. \cite[Corollary 3.4]{Collins-Sniady-2006}).

\begin{theorem} 
Let $U= (U_{ij})_{1\ls i,j\ls n}$ be an $n\times n$ Haar-distributed unitary matrix. For four sequences
${\bm i} = (i_1,\ldots,i_k)$, ${\bm i}^\prime = (i_1^\prime,\ldots,i_k^\prime)$,
${\bm j} = (j_1,\ldots,j_k)$ and ${\bm j}^\prime = (j_1^\prime,\ldots,j_k^\prime)$ of positive integers
in $[n]$, we have
\begin{equation*}
\int_{U(n)} U_{i_1j_1}U_{i_2j_2}\cdots U_{i_kj_k}\overline{U_{i_1^\prime j_1^\prime}U_{i_2^\prime j_2^\prime}\cdots U_{i_k^\prime j_k^\prime}}\,dU = \sum_{\sigma,\tau\in S_k} \delta_{\sigma}({\bm i},{\bm i}^\prime)\delta_{\tau}({\bm j},{\bm j}^\prime)\,{\rm Wg}^U(\sigma^{-1}\tau;n).
\end{equation*}
\end{theorem}

\end{remark}

\subsubsection{The orthogonal case}

For every $\sigma\in S_{2k}$ we can consider an undirected graph $G(\sigma)$ with vertices $1,2,\ldots,2k$ and edge set consisting of
\begin{equation*} \bigl\{\{2i-1,2i\} : i=1,2,\ldots,k\bigr\}\cup\bigl\{\{\sigma(2i-1),\sigma(2i)\} : i=1,2,\ldots,k\bigr\}\end{equation*}
(note that we consider as different every two edges of the form $\{2i-1,2i\}$ and $\{\sigma(2j-1),\sigma(2j)\}$ even if the sets coincide).
Then each vertex lies on exactly two edges, and the number of vertices in each connected component is even. If the numbers of vertices
in the connected components are $2\mu_1\gs 2\mu_2\gs \cdots \gs 2\mu_l$, then the sequence $\mu = (\mu_1,\mu_2,\ldots,\mu_l)$ is a partition of $k$
which is called the coset-type of $\sigma$. 

Let $M_{2k}$ be the set of all pair partitions of the set $[2k]=\{1,\ldots,2k\}$. A pair partition $\sigma \in M_{2k}$ can be uniquely expressed
in the form
\begin{equation*}\sigma = \bigl\{\{\sigma(1),\sigma(2)\},\{\sigma(3),\sigma(4)\},\ldots,\{\sigma(2k-1),\sigma(2k)\}\bigr\} \end{equation*}
where $1=\sigma(1) < \sigma(3)<\cdots <\sigma(2i-1)<\cdots<\sigma(2k-1)$ and $\sigma(2i-1) < \sigma(2i)$ for every $1\ls i\ls k$.
Then $\sigma$ can also be regarded as a permutation $\left(\begin{array}{cccc}1&2&\cdots&2k\\ \sigma(1)&\sigma(2)&\cdots&\sigma(2k)\end{array}\right)$
in $S_{2k}$. In this way we can embed $M_{2k}$ into $S_{2k}$ (in particular, we can talk about the coset-type of a pair partition $\sigma\in M_{2k}$).

\smallskip

For a permutation $\sigma\in S_{2k}$ and a $2k$-tuple ${\bm i}=(i_1,i_2,\ldots,i_{2k})$ of positive integers, set
\begin{equation} \delta_\sigma^\prime({\bm i}) = \prod_{s=1}^k\delta_{i_{\sigma(2s-1)},i_{\sigma(2s)}}.\end{equation}
In particular, if $\sigma\in M_{2k}$, then we can more simply write $\delta_\sigma^\prime({\bm i}) = \prod\limits_{\{a,b\}\in \sigma} \delta_{i_a,i_b}$.

Given a square matrix $A$ and $\sigma\in S_{2k}$ with coset-type $\mu = (\mu_1,\mu_2,\ldots,\mu_l)$, set
\begin{equation} {\rm Tr}_\sigma^\prime(A) = \prod_{j=1}^l{\rm Tr}(A^{\mu_j})\,.\end{equation}
Finally, given a partition $\lambda$ of $k$ and a number $z\in {\mathbb C}$, define
\begin{equation} C_\lambda^\prime(z) = \prod_{i=1}^{l(\lambda)}\prod_{j=1}^{\lambda_i}(z+2j-i-1) \end{equation}
(again $C_\lambda^\prime(n) = J_\lambda^2(1^n)$, see \cite[Theorem 5.4]{Stanley-1989}).

\medskip

To be able to give the analogous definition for the \emph{orthogonal Weingarten function} 
to the one we gave above in the unitary case, we need first to recall the definition 
of the \emph{zonal spherical functions} on $S_{2k}$. Let $H_k$ be the hyperoctahedral group of order $2^kk!$;
this can be realised as the subgroup of $S_{2k}$ generated by adjacent tranpositions $(2i\!-\!1 \ \ 2i)$ for any $1\ls i\ls k$ and double transpositions
of the form $(2i\!-\!1 \ \ 2j\!-\!1)(2i \  \ 2j)$ for any $1\ls i < j\ls k$. Then for each partition $\lambda$ of $k$,
consider the partition $2\lambda = (2\lambda_1,2\lambda_2,\ldots,2\lambda_{l(\lambda)})$ of $2k$ 
and the corresponding character $\chi^{2\lambda}$ of $S_{2k}$, and define the zonal spherical function 
$\omega^\lambda$
corresponding to $\lambda$ by
\begin{equation} \sigma\in S_{2k} \quad \mapsto \quad \omega^\lambda(\sigma) : = \frac{1}{2^kk!}\bigl(\chi^{2\lambda}\ast {\bf 1}_{H_k}\bigr)(\sigma)
= \frac{1}{2^kk!}\, \sum_{\pi\in S_{2k}}\chi^{2\lambda}(\sigma\pi)\ {\bf 1}_{H_k}\bigl(\pi^{-1}\bigr). \end{equation}
Given that $H_k$ is a subgroup of $S_{2k}$ and that $M_{2k}$ contains a unique representative of each 
{\bf left} coset $\sigma H_k$ of $H_k$ in $S_{2k}$,
this definition can be rewritten in a somewhat simpler way:
\begin{equation}\omega^\lambda(\sigma) = \frac{1}{2^kk!}\,\sum_{\tau\in M_{2k}}\,\sum_{\zeta\in H_k}\chi^{2\lambda}\bigl(\sigma\,\tau\,\zeta\bigr)\ {\bf 1}_{H_k}\bigl((\tau\zeta)^{-1}\bigr) 
= \frac{1}{2^kk!}\sum_{\zeta\in H_k}\chi^{2\lambda}(\sigma\zeta).\end{equation}
Recall finally that the zonal sperical functions $\omega^\lambda$ corresponding to partitions $\lambda$ of $k$ form a linear basis
of $L(S_{2k},H_k)$, the space of all complex-valued functions on $S_{2k}$ which are $H_k$-bi-invariant, that is, the set
\begin{equation*} \bigl\{f: S_{2k}\to {\mathbb C}\mid f(\zeta\sigma) = f(\sigma\zeta) = f(\sigma)\ \hbox{for every}\ \sigma \in S_{2k}, \zeta\in H_k\bigr\}. \end{equation*}

\bigskip

We now define the orthogonal Weingarten function on $S_{2k}$ with one complex parameter $z\in {\mathbb C}$ 
(see \cite{Collins-Matsumoto-2009} or \cite{Collins-Matsumoto-Saad-2014}):
\begin{equation} \sigma\in S_{2k} \quad \mapsto \quad 
{\rm Wg}^O(\sigma;z) := \frac{2^kk!}{(2k)!}\sum_{\substack{\lambda\vdash k\\C_\lambda^\prime(z)\neq 0}}\frac{\chi^{2\lambda}(e)}{C_\lambda^\prime(z)}\,\omega^\lambda(\sigma).  \end{equation}
Note that all $\omega^\lambda$, and therefore also $Wg^O(\cdot;z)$, take the same value
at permutations $\sigma_1, \sigma_2$ with the same coset-type (where equivalently 
$\sigma_1$ has the same coset-type as $\sigma_2$ if and only if $\sigma_1\in H_k\sigma_2 H_k$). 

\begin{theorem}{\rm (Conjugacy invariance, \cite[Theorem 3.3]{Collins-Matsumoto-Saad-2014})} 
\label{thm:CMS-R-conj-inv}
Let $T = (T_{ij})$ be an $n\times n$ real symmetric random matrix with the invariance property that
$OTO^t$ has the same distribution as $T$ for any orthogonal matrix $O$. 
For any sequence ${\bm i} = (i_1,\ldots,i_{2k})$, we have
\begin{equation*} {\mathbb E}\bigl[T_{i_1i_2}T_{i_3i_4}\cdots T_{i_{2k-1}i_{2k}}\bigr] = 
\sum_{\sigma,\tau\in M_{2k}} \delta_\sigma^\prime({\bm i})\,{\rm Wg}^O(\sigma^{-1}\tau;n)\ {\mathbb E}\bigl[{\rm Tr}_\tau^\prime(T)\bigr].
\end{equation*}
\end{theorem}

\begin{theorem}{\rm (Left-right invariance, \cite[Theorem 3.5]{Collins-Matsumoto-Saad-2014})} 
\label{thm:CMS-R-lr-inv}
Let $X$ be a real $n\times p$ random matrix which has the same distribution as $OXQ$
for any orthogonal matrices $O,Q$. Then, for two sequences ${\bm i} = (i_1,\ldots,i_{2k})$
and ${\bm j} = (j_1,\ldots,j_{2k})$, we have
\begin{equation*} {\mathbb E}\bigl[X_{i_1j_1}X_{i_2j_2}\cdots X_{i_{2k}j_{2k}}\bigr] = 
\sum_{\sigma_1,\sigma_2,\tau_1,\tau_2\in M_{2k}} \delta_{\sigma_1}^\prime({\bm i})\delta_{\sigma_2}^\prime({\bm j})\,
{\rm Wg}^O(\sigma_1^{-1}\tau_1;n)\,{\rm Wg}^O(\sigma_2^{-1}\tau_2;p)\ {\mathbb E}[{\rm Tr}_{\tau_1^{-1}\tau_2}^\prime(XX^t)].
\end{equation*}
\end{theorem}

\begin{remark}
Again the proof of the theorems follows from a decomposition of $T$ or $X$ as $ODO^t$
or $ODQ^t$ respectively (with $D$ diagonal with the same distribution of eigenvalues or singular values
as $T$ or $X$ respectively, $O$ and $Q$ Haar-distributed random orthogonal matrices, and $D, O$ and $Q$
independent), combined with the use of the following result (see \cite[Corollary 3.4]{Collins-Sniady-2006} 
and \cite{Collins-Matsumoto-2009}).

\begin{theorem}
Let $O= (O_{ij})_{1\ls i,j\ls n}$ be an $n\times n$ Haar-distributed orthogonal matrix. For sequences
${\bm i} = (i_1,\ldots,i_{2k})$, ${\bm j} = (j_1,\ldots,j_{2k})$ of positive integers in $[n]$, we have
\begin{equation*}
\int_{O(n)} O_{i_1j_1}O_{i_2j_2}\cdots O_{i_{2k}j_{2k}}\,dO = \sum_{\sigma,\tau\in M_{2k}} \delta_{\sigma}^\prime({\bm i})\delta_{\tau}^\prime({\bm j})\,{\rm Wg}^O(\sigma^{-1}\tau;n).
\end{equation*}
\end{theorem}

\noindent
Note that the statement of Theorem \ref{thm:CMS-R-lr-inv} above is slightly different from 
that in \cite{Collins-Matsumoto-Saad-2014}, the conclusion following from the proof
on \cite[p. 9]{Collins-Matsumoto-Saad-2014}, and being compatible 
with the invariances of ensembles such as 
$X\sim {\rm Unif}(K_{p, {\cal M}_n({\mathbb R})})$ under taking transpose.
\end{remark}

\section{Proof of Theorem \ref{thm:var-operator}}\label{sec:main-proof}

Let us start with the case where $E = {\cal M}_n({\mathbb F})$. By Proposition \ref{prop:var-reduction}
it suffices to show that
\begin{equation*} {\rm Var}_{N_\infty}\bigl(\|x\|_2^2\bigr) =  \frac{N_\infty\bigl(\|x\|_2^4\bigr)}{N_\infty(1)} - 
\left(\frac{N_\infty\bigl(\|x\|_2^2\bigr)}{N_\infty(1)}\right)^2 \simeq 1, \end{equation*}
where in this case
\begin{equation*} N_\infty(f) = \int\limits_{[-1,1]^n} f(x)\ \cdot \!\!\! \prod_{1\ls i<j\ls n}\big|\,x_i^2 - x_j^2\big|^\beta\,\cdot \!\!\!\prod_{1\ls i\ls n}|x_i|^{\beta-1}\,dx \end{equation*}
with $\beta= {\rm dim}_{{\mathbb R}}({\mathbb F})$. Since all the functions $f$ we need to consider
are symmetric and in addition their values only depend on what the absolute values of the coordinates of
their input are, we have
\begin{equation*} {\rm Var}_{N_\infty}\bigl(\|x\|_2^2\bigr) = \frac{\tilde{N}_\infty\bigl(\|x\|_2^4\bigr)}{\tilde{N}_\infty(1)} - 
\left(\frac{\tilde{N}_\infty\bigl(\|x\|_2^2\bigr)}{\tilde{N}_\infty(1)}\right)^2  \end{equation*}
where
\begin{equation*}
\tilde{N}_\infty(f):= \int\limits_{[0,1]^n} f(x)\ \cdot \!\!\! \prod_{1\ls i<j\ls n}\big|\,x_i^2 - x_j^2\big|^\beta\,\cdot \!\!\!\prod_{1\ls i\ls n}|x_i|^{\beta-1}\,dx = \frac{1}{2^n} N_\infty(f)
\end{equation*}
for all the functions considered. Furthermore, by symmetry again,
\begin{equation} \label{eqp:MnF-variance-exp}
{\rm Var}_{N_\infty}\bigl(\|x\|_2^2\bigr) = n\frac{\tilde{N}_\infty\bigl(x_1^4\bigr)}{\tilde{N}_\infty(1)}
+n(n-1)\frac{\tilde{N}_\infty\bigl(x_1^2x_2^2\bigr)}{\tilde{N}_\infty(1)} 
- n^2 \left(\frac{\tilde{N}_\infty\bigl(x_1^2\bigr)}{\tilde{N}_\infty(1)}\right)^2.
\end{equation}

Employing now the transformation $x = (x_1,x_2,\ldots, x_n) \in [0,1]^n \mapsto \bigl(\sqrt{x_1},\sqrt{x_2},\ldots,\sqrt{x_n}\bigr)$
which has Jacobian $x\in (0,1)^n \mapsto 2^{-n}\prod_ix_i^{-1/2}$,
we can obtain the following:
\begin{gather*}
\tilde{N}_\infty(1) = 2^{-n} \int\limits_{[0,1]^n} \prod_{1\ls i\ls n} x_i^{\frac{\beta}{2}-1}\cdot
\prod_{1\ls i<j\ls n}\big|\,x_i - x_j\big|^{\beta}dx = 2^{-n} \, I_0\Bigl(n; \frac{\beta}{2}, 1, \frac{\beta}{2}\Bigr),
\\
\tilde{N}_\infty(x_1^2) = 2^{-n}\int\limits_{[0,1]^n} x_1\prod_{1\ls i\ls n} x_i^{\frac{\beta}{2}-1}\cdot
\prod_{1\ls i<j\ls n}\big|\,x_i - x_j\big|^{\beta}dx = 2^{-n} \, I_1\Bigl(n; \frac{\beta}{2}, 1, \frac{\beta}{2}\Bigr),
\\
\tilde{N}_\infty\bigl(x_1^2x_2^2\bigr) =  2^{-n}\int\limits_{[0,1]^n} x_1x_2\prod_{1\ls i\ls n} x_i^{\frac{\beta}{2}-1}\cdot
\prod_{1\ls i<j\ls n}\big|\,x_i - x_j\big|^{\beta}dx = 2^{-n}\, \Bigl(I_1\Bigl(n; \frac{\beta}{2}, 1, \frac{\beta}{2}\Bigr)- I_{1, 1,0}\Bigl(n; \frac{\beta}{2}, 1, \frac{\beta}{2}\Bigr)\Bigr),
\\
\intertext{and finally}
\tilde{N}_\infty\bigl(x_1^4\bigr) =  2^{-n}\int\limits_{[0,1]^n} x_1^2\prod_{1\ls i\ls n} x_i^{\frac{\beta}{2}-1}\cdot
\prod_{1\ls i<j\ls n}\big|\,x_i - x_j\big|^{\beta}dx = 2^{-n}\, \Bigl(I_1\Bigl(n; \frac{\beta}{2}, 1, \frac{\beta}{2}\Bigr)- I_{1, 1,1}\Bigl(n; \frac{\beta}{2}, 1, \frac{\beta}{2}\Bigr)\Bigr)
\end{gather*}
(recall the notation in Subsection \ref{subsec:SAK-results}). Using the formulas in \eqref{eq1:Aomoto-integrals}
and \eqref{eq3:Aomoto-integrals}, we see that
\begin{gather*}
\frac{\tilde{N}_\infty\bigl(x_1^2\bigr)}{\tilde{N}_\infty(1)} = 
\frac{n\beta/2}{1+(2n-1)\beta/2} 
\\
\frac{\tilde{N}_\infty\bigl(x_1^2x_2^2\bigr)}{\tilde{N}_\infty(1)} =
\frac{n\beta/2}{1+(2n-1)\beta/2} - \frac{n\beta/2(1+(n-1)\beta/2)}{(1+(2n-1)\beta/2)(1+(n-1)\beta)} 
= \frac{n(n-1) \beta^2/4}{(1+(2n-1)\beta/2)(1+(n-1)\beta)},
\end{gather*}
\begin{align*}
 \frac{\tilde{N}_\infty\bigl(x_1^4\bigr)}{\tilde{N}_\infty(1)} &= 
\frac{n\beta/2}{1+(2n-1)\beta/2} \ - \  \frac{1+(n-1)\beta/2}{2+ (2n-1)\beta/2}\cdot \frac{n\beta/2(1+(n-1)\beta/2)}{(1+(2n-1)\beta/2)(1+(n-1)\beta)}
\\
& = \frac{n\beta/2(1/2 +3(n-1)\beta/4)}{(1+(2n-1)\beta/2)(1+(n-1)\beta)} + \frac{n\beta^2/8(1+(n-1)\beta/2)}{(2+ (2n-1)\beta/2)(1+(2n-1)\beta/2)(1+(n-1)\beta)}.
\end{align*}

Plugging these into \eqref{eqp:MnF-variance-exp}, we deduce that
\begin{align*} 
n\frac{\tilde{N}_\infty\bigl(x_1^4\bigr)}{\tilde{N}_\infty(1)}
\, +\, & n(n-1)\frac{\tilde{N}_\infty\bigl(x_1^2x_2^2\bigr)}{\tilde{N}_\infty(1)} \,-\,
n^2 \left(\frac{\tilde{N}_\infty\bigl(x_1^2\bigr)}{\tilde{N}_\infty(1)}\right)^2
\\ =\, & \frac{n^4\beta^2/4 - n^3\beta^2/8 +n^2\beta/2(1/2-\beta/4)}{(1+(2n-1)\beta/2)(1+(n-1)\beta)}
+  \frac{n^3\beta^3/16 + n^2\beta^2/8(1-\beta/2)}{(2+ (2n-1)\beta/2)(1+(2n-1)\beta/2)(1+(n-1)\beta)}
\\ &\,- \frac{n^4\beta^2/4}{(1+(2n-1)\beta/2)(1+(n-1)\beta)} \left(1 - \frac{\beta/2}{1+(2n-1)\beta/2}\right)
\\ = \,& \frac{n^2\beta/2(1/2-\beta/4)}{(1+(2n-1)\beta/2)(1+(n-1)\beta)}
+  \frac{n^3\beta^3/16 + n^2\beta^2/8(1-\beta/2)}{(2+ (2n-1)\beta/2)(1+(2n-1)\beta/2)(1+(n-1)\beta)}
\\ &\,+\frac{n^3\beta^2/8(\beta/2-1)}{(1+(2n-1)\beta/2)^2(1+(n-1)\beta)}
\\ = & \frac{n^3\beta^2/8 + n^2\beta/2\bigl((1/2-\beta/4)(2-\beta/2)+\beta/4(1-\beta/2)\bigr)}{(2+ (2n-1)\beta/2)(1+(2n-1)\beta/2)(1+(n-1)\beta)}  
+ \frac{n^3\beta^2/8(\beta/2-1)}{(2+(2n-1)\beta/2)(1+(2n-1)\beta/2)^2(1+(n-1)\beta)}
\\ =\, & \frac{1}{8\beta} + O\Bigl(\frac{1}{n}\Bigr).
\end{align*}
This agrees with the conclusion of \cite[Theorem 1]{Radke-V-2016} (see more specifically the end of Section 4 in \cite{Radke-V-2016}).

\bigskip

We now turn to the cases of the subspaces of ${\mathbb F}$-self-adjoint matrices. 
Recall that by Proposition \ref{prop:var-reduction} it suffices to show 
\begin{align} 
\nonumber {\rm Var}_{N_\infty}\bigl(\|x\|_2^2\bigr) &=  \frac{N_\infty\bigl(\|x\|_2^4\bigr)}{N_\infty(1)} - 
\left(\frac{N_\infty\bigl(\|x\|_2^2\bigr)}{N_\infty(1)}\right)^2 
\\ \label{eqp:AdjF-variance-exp}
&= n\frac{N_\infty\bigl(x_1^4\bigr)}{N_\infty(1)}
+n(n-1)\frac{N_\infty\bigl(x_1^2x_2^2\bigr)}{N_\infty(1)} 
- n^2 \left(\frac{N_\infty\bigl(x_1^2\bigr)}{N_\infty(1)}\right)^2 \simeq 1, \end{align}
where now
\begin{equation*} N_\infty(f) = \int\limits_{[-1,1]^n} f(x)\ \cdot \!\!\!\! \prod_{1\ls i<j\ls n}\big|\,x_i - x_j\big|^\beta\,dx \end{equation*}
with $\beta= {\rm dim}_{{\mathbb R}}({\mathbb F})$. For each of the functions $f$ in \eqref{eqp:AdjF-variance-exp}
we can write
\begin{align*}
N_\infty(f) &= \int\limits_{[-\frac{1}{2},\frac{1}{2}]^n} 2^n f(2x_1,\ldots,2x_n)\ \cdot \!\!\!\! \prod_{1\ls i<j\ls n}\big|\,2x_i - 2x_j\big|^\beta\,dx
\\
& = 2^{n+\beta n(n-1)/2 + s}\int\limits_{[-\frac{1}{2},\frac{1}{2}]^n} f(x_1,\ldots,x_n)\ \cdot \!\!\!\! \prod_{1\ls i<j\ls n}\big|\,x_i - x_j\big|^\beta\,dx
\\
&= 2^{n+\beta n(n-1)/2 + s}\int\limits_{[0,1]^n} f\Bigl(t_1-\frac{1}{2},\ldots,t_n-\frac{1}{2}\Bigr)\ \cdot \!\!\!\! \prod_{1\ls i<j\ls n}\big|\,t_i - t_j\big|^\beta\,d{\bm t},
\end{align*}
where $s$ is the degree of homogeneity of $f$. Thus, upon writing
\begin{equation*} 
J_\infty(g) = \int\limits_{[0,1]^n} g({\bm t}) \cdot \!\!\!\! \prod_{1\ls i<j\ls n}\big|\,t_i - t_j\big|^\beta\,d{\bm t},
\end{equation*}
we see that, to verify \eqref{eqp:AdjF-variance-exp}, we need to estimate
\begin{align*}
J_\infty\Bigl(\Bigl(t_1-\frac{1}{2}\Bigr)^2\Bigr) &=
J_\infty(t_1^2) - J_\infty(t_1) + \frac{1}{4}J_\infty(1) 
= \frac{1}{n} J_\infty(m_{(2)}) -\frac{1}{n} J_\infty(m_{(1)}) + \frac{1}{4}J_\infty(1),
\\
J_\infty\Bigl(\Bigl(t_1-\frac{1}{2}\Bigr)^2\Bigl(t_2-\frac{1}{2}\Bigr)^2\Bigr) &=
J_\infty(t_1^2t_2^2) - J_\infty(t_1^2t_2+t_1t_2^2) +\frac{1}{4}J_\infty(t_1^2+t_2^2) + J_\infty(t_1t_2) -\frac{1}{4}J_\infty(t_1+t_2)
+\frac{1}{16}J_\infty(1)
\\
& = \frac{2}{n(n-1)}J_\infty(m_{(2^2)}) - \frac{2}{n(n-1)}J_\infty(m_{(2,1)}) +\frac{1}{2n}J_\infty(m_{(2)})
\\ & \hspace{3.2cm} +\frac{2}{n(n-1)}J_\infty(m_{(1^2)}) - \frac{1}{2n}J_\infty(m_{(1)}) +\frac{1}{16}J_\infty(1),
\\
J_\infty\Bigl(\Bigl(t_1-\frac{1}{2}\Bigr)^4\Bigr) &=
J_\infty(t_1^4) - 2J_\infty(t_1^3) +\frac{3}{2}J_\infty(t_1^2) -\frac{1}{2}J_\infty(t_1)
+\frac{1}{16}J_\infty(1)
\\
& = \frac{1}{n} J_\infty(m_{(4)})- \frac{2}{n} J_\infty(m_{(3)}) + \frac{3}{2n} J_\infty(m_{(2)}) - \frac{1}{2n} J_\infty(m_{(1)})
+\frac{1}{16}J_\infty(1).
\end{align*}

We will do so by recalling 
the decompositions of the monomial symmetric functions 
in the bases of the Schur or the zonal or the quaternionic zonal polynomials
(see tables \eqref{conversion-table1} and \eqref{conversion-table2}),
and by using integration formula \eqref{eq:Kadell-integration-formula}.
Denote by $I_n^\kappa(\lambda)$ the integral
\begin{equation*} 
\int\limits_{[0,1]^n} P_\lambda^{1/\kappa}({\bm t})
\prod_{1\ls i<j\ls n}\big|\,t_i - t_j\big|^{2\kappa}\,d{\bm t} =
\int\limits_{[0,1]^n} s_\lambda^\kappa({\bm t})
\prod_{1\ls i<j\ls n}\big|\,t_i - t_j\big|^{2\kappa}\,d{\bm t}.
\end{equation*}
For simplicity and to make it easier to check the tedious computations, 
in what follows we treat the cases of ${\mathbb C}, {\mathbb R}$ and ${\mathbb H}$ separately
(note moreover that, even though the below computations could be done for more general values of $\beta$
(see Remark \ref{rem:beta-ensembles-estimates}),
and would still have an interpretation via a random matrix model (see \cite{Edelman-Sutton-2008}),
this interpretation would not correspond to the same type of 
variance problem as the one we are interested in here).

\begin{proposition}(Case of $\beta=2$, $\kappa=1$; Hermitian matrices) \label{prop:C-trace-estimates}
The following estimates are true:
\begin{gather*}
\frac{N_\infty\bigl(x_1^2\bigr)}{N_\infty(1)} = 2\frac{J_\infty\Bigl(\Bigl(t_1-\frac{1}{2}\Bigr)^2\Bigr)}{J_\infty(1)}
%= 2\left(\frac{1}{8} - \frac{1}{32n^2} + O\Bigl(\frac{1}{n^3}\Bigr)\right),
= \frac{1}{4} - \frac{1}{16n^2} + O\Bigl(\frac{1}{n^3}\Bigr),
%\\
%\frac{N_\infty\bigl(x_1x_2\bigr)}{N_\infty(1)} = 2\frac{J_\infty\Bigl(\Bigl(t_1-\frac{1}{2}\Bigr) \Bigl(t_2-\frac{1}{2}\Bigr)\Bigr)}{J_\infty(1)}
%%= 2\left(- \frac{1}{8n} - \frac{1}{16n^2} -\frac{1}{32n^3} + O\Bigl(\frac{1}{n^4}\Bigr)\right),
%= -\frac{1}{4n} - \frac{1}{8n^2} -\frac{1}{16n^3} + O\Bigl(\frac{1}{n^4}\Bigr),
\\
\frac{N_\infty\bigl(x_1^2x_2^2\bigr)}{N_\infty(1)} = 4\frac{J_\infty\Bigl(\Bigl(t_1-\frac{1}{2}\Bigr)^2\Bigl(t_2-\frac{1}{2}\Bigr)^2\Bigr)}{J_\infty(1)}
%= 4\left(\frac{1}{64} - \frac{1}{128n} - \frac{1}{128n^2} + O\Bigl(\frac{1}{n^3}\Bigr)\right)
= \frac{1}{16} - \frac{1}{32n} - \frac{1}{32n^2} + O\Bigl(\frac{1}{n^3}\Bigr)
\\
\intertext{and}
\frac{N_\infty\bigl(x_1^4\bigr)}{N_\infty(1)} = 4\frac{J_\infty\Bigl(\Bigl(t_1-\frac{1}{2}\Bigr)^4\Bigr)}{J_\infty(1)}
%= 4\left(\frac{3}{128}  + O\Bigl(\frac{1}{n^2}\Bigr)\right).
= \frac{3}{32}  + O\Bigl(\frac{1}{n^2}\Bigr).
\end{gather*}
As a consequence,
\begin{equation*}
{\rm Var}_{N_\infty}\bigl(\|x\|_2^2\bigr) = n\frac{N_\infty\bigl(x_1^4\bigr)}{N_\infty(1)}
+n(n-1)\frac{N_\infty\bigl(x_1^2x_2^2\bigr)}{N_\infty(1)} 
- n^2 \left(\frac{N_\infty\bigl(x_1^2\bigr)}{N_\infty(1)}\right)^2 = \frac{1}{32} + O\Bigl(\frac{1}{n}\Bigr).
\end{equation*}
Moreover, 
\begin{equation*}
\\
\frac{N_\infty\bigl(x_1x_2\bigr)}{N_\infty(1)} = 2\frac{J_\infty\Bigl(\Bigl(t_1-\frac{1}{2}\Bigr) \Bigl(t_2-\frac{1}{2}\Bigr)\Bigr)}{J_\infty(1)}
%= 2\left(- \frac{1}{8n} - \frac{1}{16n^2} -\frac{1}{32n^3} + O\Bigl(\frac{1}{n^4}\Bigr)\right),
= -\frac{1}{4n} - \frac{1}{8n^2} -\frac{1}{16n^3} + O\Bigl(\frac{1}{n^4}\Bigr)
\end{equation*}
(this is an estimate we will need in the following section).
\end{proposition}
\begin{proof}
We begin with the simple observation that for all $\kappa$ we have 
\begin{equation*}
J_\infty(1) =  I_n^\kappa((0)) \qquad \hbox{and} \qquad J_\infty(m_{(1)}) = I_n^\kappa((1)) = \frac{n}{2}I_n^\kappa((0)).
\end{equation*}
Furthermore, when $\kappa = 1$,
\begin{align*}
J_\infty(m_{(1^2)}) = I_n^1((1^2)) = I_n^1((0))\frac{n(n-1)}{4}\frac{n-1}{2n-1}
&= I_n^1((0))\,\frac{n(n-1)}{2}\, \Bigl(\frac{1}{4} - \frac{1}{8n} - \frac{1}{16n^2} -\frac{1}{32n^3} + O\Bigl(\frac{1}{n^4}\Bigr)\Bigr)
\\ &= I_n^1((0))\,n\,\Bigl(\frac{n}{8} - \frac{3}{16} + \frac{1}{32n} + \frac{1}{64n^2} + O\Bigl(\frac{1}{n^3}\Bigr)\Bigr)
\\ \hbox{and} \qquad I_n^1((2)) = I_n^1((0))\frac{n(n+1)}{4}\frac{n+1}{2n+1} 
&= I_n^1((0))\,n\,\Bigl(\frac{n}{8} + \frac{3}{16} + \frac{1}{32n} - \frac{1}{64n^2} + O\Bigl(\frac{1}{n^3}\Bigr)\Bigr).
\end{align*}
Therefore,
\begin{equation*}
J_\infty(m_{(2)}) = I_n^1((2)) - I_n^1((1^2))
= I_n^1((0))\,n\,\Bigl(\frac{3}{8} - \frac{1}{32n^2} + O\Bigl(\frac{1}{n^3}\Bigr)\Bigr),
\end{equation*}
which also gives
\begin{equation} \label{eqp:C-second-moment}
J_\infty\Bigl(\Bigl(t_1-\frac{1}{2}\Bigr)^2\Bigr) 
=  \frac{1}{n} J_\infty(m_{(2)}) -\frac{1}{n} J_\infty(m_{(1)}) + \frac{1}{4}J_\infty(1)
= I_n^1((0))\,\Bigl(\frac{1}{8} - \frac{1}{32n^2} + O\Bigl(\frac{1}{n^3}\Bigr)\Bigr).
\end{equation}
Note also that 
\begin{equation} \label{eqp:C-linear-cross-term}
J_\infty\Bigl(\Bigl(t_1-\frac{1}{2}\Bigr)\Bigl(t_2-\frac{1}{2}\Bigr)\Bigr) 
= \frac{2}{n(n-1)} J_\infty(m_{(1^2)}) - \frac{1}{n} J_\infty(m_{(1)}) + \frac{1}{4}J_\infty(1)
= I_n^1((0))\,\Bigl(- \frac{1}{8n} - \frac{1}{16n^2} -\frac{1}{32n^3} + O\Bigl(\frac{1}{n^4}\Bigr)\Bigr).
\end{equation}

Next observe that
\begin{align*}
I_n^1((1^3)) = I_n^1((0))\frac{n(n-1)(n-2)}{24}\frac{n-2}{2n-1}
&= I_n^1((0))\,\frac{n(n-1)}{2}\,\Bigl(\frac{n}{24} - \frac{7}{48} + \frac{3}{32n} + \frac{3}{64n^2}  + O\Bigl(\frac{1}{n^3}\Bigr)\Bigr)
\\ 
& = I_n^1((0))\,n\,\Bigl(\frac{n^2}{48} - \frac{3n}{32}+ \frac{23}{192} - \frac{3}{128n}  + O\Bigl(\frac{1}{n^2}\Bigr)\Bigr),
\end{align*}
\begin{align*}
I_n^1((2,1)) = I_n^1((0))\frac{n(n-1)(n+1)}{6}\frac{n+1}{2n+1}\frac{n-1}{2n-1}
&= I_n^1((0))\,\frac{n(n-1)}{2}\,\Bigl(\frac{n}{12} + \frac{1}{12} - \frac{1}{16n} - \frac{1}{16n^2}  + O\Bigl(\frac{1}{n^3}\Bigr)\Bigr)
\\ 
& = I_n^1((0))\,n\,\Bigl(\frac{n^2}{24} - \frac{7}{96}  + O\Bigl(\frac{1}{n^2}\Bigr)\Bigr)
\end{align*}
and 
\begin{equation*}
I_n^1((3)) = I_n^1((0))\frac{(n+2)(n+1)n}{24}\frac{n+2}{2n+1}
= I_n^1((0))\,n\,\Bigl(\frac{n^2}{48} + \frac{3n}{32} + \frac{23}{192} + \frac{3}{128n} + O\Bigl(\frac{1}{n^2}\Bigr)\Bigr).
\end{equation*}
It follows that
\begin{equation*}
J_\infty(m_{(2,1)}) = I_n^1((2,1)) - 2I_n^1((1^3))
= I_n^1((0))\,\frac{n(n-1)}{2}\,\Bigl(\frac{3}{8} - \frac{1}{4n} - \frac{5}{32n^2}  + O\Bigl(\frac{1}{n^3}\Bigr)\Bigr)
\end{equation*}
and
\begin{equation*} 
J_\infty(m_{(3)}) = I_n^1((3)) - I_n^1((2,1)) + I_n^1((1^3)) 
= I_n^1((0))\,n\,\Bigl(\frac{5}{16} + O\Bigl(\frac{1}{n^2}\Bigr)\Bigr).
\end{equation*}

Moreover,
\begin{align*}
I_n^1((1^4)) = I_n^1((0))\frac{n(n-1)(n-2)(n-3)}{96}\frac{n-2}{2n-1}\frac{n-3}{2n-3}
&= I_n^1((0))\,\frac{n(n-1)}{2}\,\Bigl(\frac{n^2}{192} - \frac{n}{24} + \frac{27}{256} - \frac{9}{128n} - \frac{33}{1024n^2}  + O\Bigl(\frac{1}{n^3}\Bigr)\Bigr)
\\ 
& = I_n^1((0))\,n\,\Bigl(\frac{n^3}{384} - \frac{3n^2}{128}  + \frac{113n}{1536} - \frac{45}{1536} + \frac{39}{2048n}  + O\Bigl(\frac{1}{n^2}\Bigr)\Bigr),
\end{align*}
\begin{align*}
I_n^1((2,1^2)) = I_n^1((0))\frac{(n+1)n(n-1)(n-2)}{32}\frac{n+1}{2n+1}\frac{n-2}{2n-1}
&= I_n^1((0))\,\frac{n(n-1)}{2}\,\Bigl(\frac{n^2}{64} - \frac{n}{32} - \frac{11}{256} + \frac{7}{128n} + \frac{53}{1024n^2}  + O\Bigl(\frac{1}{n^3}\Bigr)\Bigr)
\\ 
& = I_n^1((0))\,n\,\Bigl(\frac{n^3}{128} - \frac{3n^2}{128} - \frac{3n}{512} + \frac{25}{512} - \frac{3}{2048n} + O\Bigl(\frac{1}{n^2}\Bigr)\Bigr),
\end{align*}
\begin{align*}
I_n^1((2^2)) = I_n^1((0))\frac{n(n-1)(n+1)n}{48}\frac{n+1}{2n+1}\frac{n-1}{2n-1}
&= I_n^1((0))\,\frac{n(n-1)}{2}\,\Bigl(\frac{n^2}{96} + \frac{n}{96} - \frac{1}{128} - \frac{1}{128n} - \frac{1}{512n^2}  + O\Bigl(\frac{1}{n^3}\Bigr)\Bigr)
\\ 
& = I_n^1((0))\,n\,\Bigl(\frac{n^3}{192} - \frac{7n}{768} + \frac{3}{1024n} + O\Bigl(\frac{1}{n^2}\Bigr)\Bigr),
\end{align*}
%%%
while
\begin{equation*}
I_n^1((3,1)) = I_n^1((0))\frac{(n+2)(n+1)n(n-1)}{32}\frac{n+2}{2n+1}\frac{n-1}{2n-1}
= I_n^1((0))\,n\,\Bigl(\frac{n^3}{128} + \frac{3n^2}{128} - \frac{3n}{512} - \frac{25}{512} - \frac{3}{2048n} + O\Bigl(\frac{1}{n^2}\Bigr)\Bigr)
\end{equation*}
%%%%
and 
\begin{equation*}
I_n^1((4)) = I_n^1((0))\frac{(n+3)(n+2)(n+1)n}{96}\frac{n+3}{2n+3}\frac{n+2}{2n+1}
= I_n^1((0))\,n\,\Bigl(\frac{n^3}{384} + \frac{3n^2}{128} + \frac{113n}{1536} + \frac{45}{512} + \frac{39}{2048n} + O\Bigl(\frac{1}{n^2}\Bigr)\Bigr).
\end{equation*}
It follows that
\begin{equation*}
J_\infty(m_{(2^2)}) = I_n^1((2^2)) - I_n^1((2,1^2)) + I_n^1((1^4))
= I_n^1((0))\,\frac{n(n-1)}{2}\,\Bigl(\frac{9}{64} - \frac{17}{128n} - \frac{11}{128n^2}  + O\Bigl(\frac{1}{n^3}\Bigr)\Bigr)
\end{equation*}
and 
\begin{equation*}
J_\infty(m_{(4)})  = I_n^1((4)) - I_n^1((3,1)) + I_n^1((2,1^2)) + I_n^1((1^4))
= I_n^1((0))\,n\,\Bigl(\frac{35}{128} + O\Bigl(\frac{1}{n^2}\Bigr)\Bigr).
\end{equation*}

We conclude that
\begin{align} 
\nonumber
J_\infty\Bigl(\Bigl(t_1-\frac{1}{2}\Bigr)^2\Bigl(t_2-\frac{1}{2}\Bigr)^2\Bigr) 
&= \frac{2}{n(n-1)}J_\infty(m_{(2^2)}) - \frac{2}{n(n-1)}J_\infty(m_{(2,1)}) +\frac{1}{2n}J_\infty(m_{(2)})
\\ \nonumber & \hspace{3.2cm} +\frac{2}{n(n-1)}J_\infty(m_{(1^2)}) - \frac{1}{2n}J_\infty(m_{(1)}) +\frac{1}{16}J_\infty(1)
\\ \label{eqp:C-square-cross-term} &= I_n^1((0))\,\Bigl(\frac{1}{64} - \frac{1}{128n} - \frac{1}{128n^2} + O\Bigl(\frac{1}{n^3}\Bigr)\Bigr),
\end{align}
while
\begin{align} 
\nonumber
J_\infty\Bigl(\Bigl(t_1-\frac{1}{2}\Bigr)^4\Bigr) 
& =  \frac{1}{n} J_\infty(m_{(4)})- \frac{2}{n} J_\infty(m_{(3)}) + \frac{3}{2n} J_\infty(m_{(2)}) - \frac{1}{2n} J_\infty(m_{(1)})
+\frac{1}{16}J_\infty(1)
\\ \label{eqp:C-fourth-moment}
& = I_n^1((0))\,\Bigl(\frac{3}{128}  + O\Bigl(\frac{1}{n^2}\Bigr)\Bigr).
\end{align}
This completes the proof of Theorem \ref{thm:var-operator} when ${\mathbb F} = {\mathbb C}$.
\end{proof}

\begin{proposition}(Case of $\beta=1$, $\kappa=\frac{1}{2}$; ${\mathbb R}$-self-adjoint matrices)
The following estimates are true:
\begin{gather*}
\frac{N_\infty\bigl(x_1^2\bigr)}{N_\infty(1)} = 2\frac{J_\infty\Bigl(\Bigl(t_1-\frac{1}{2}\Bigr)^2\Bigr)}{J_\infty(1)}
%= 2\left(\frac{1}{8} -\frac{1}{16n} + \frac{1}{32n^2} + O\Bigl(\frac{1}{n^3}\Bigr)\right),
= \frac{1}{4} - \frac{1}{8n} + \frac{1}{16n^2} + O\Bigl(\frac{1}{n^3}\Bigr),
\\
\frac{N_\infty\bigl(x_1^2x_2^2\bigr)}{N_\infty(1)} = 4\frac{J_\infty\Bigl(\Bigl(t_1-\frac{1}{2}\Bigr)^2\Bigl(t_2-\frac{1}{2}\Bigr)^2\Bigr)}{J_\infty(1)}
%= 4\left(\frac{1}{64} - \frac{3}{128n} + \frac{3}{128n^2} + O\Bigl(\frac{1}{n^3}\Bigr)\right)
= \frac{1}{16} - \frac{3}{32n} + \frac{3}{32n^2} + O\Bigl(\frac{1}{n^3}\Bigr)
\\
\intertext{and}
\frac{N_\infty\bigl(x_1^4\bigr)}{N_\infty(1)} = 4\frac{J_\infty\Bigl(\Bigl(t_1-\frac{1}{2}\Bigr)^4\Bigr)}{J_\infty(1)}
%= 4\left(\frac{3}{128} - \frac{5}{256n}  + O\Bigl(\frac{1}{n^2}\Bigr)\right).
= \frac{3}{32} - \frac{5}{64n} + O\Bigl(\frac{1}{n^2}\Bigr).
\end{gather*}
As a consequence,
\begin{equation*}
{\rm Var}_{N_\infty}\bigl(\|x\|_2^2\bigr) 
%= n\frac{N_\infty\bigl(x_1^4\bigr)}{N_\infty(1)}
%+n(n-1)\frac{N_\infty\bigl(x_1^2x_2^2\bigr)}{N_\infty(1)} 
%- n^2 \left(\frac{N_\infty\bigl(x_1^2\bigr)}{N_\infty(1)}\right)^2 
= \frac{1}{16} + O\Bigl(\frac{1}{n}\Bigr).
\end{equation*}
Moreover, 
\begin{equation*}
\\
\frac{N_\infty\bigl(x_1x_2\bigr)}{N_\infty(1)} = 2\frac{J_\infty\Bigl(\Bigl(t_1-\frac{1}{2}\Bigr) \Bigl(t_2-\frac{1}{2}\Bigr)\Bigr)}{J_\infty(1)}
%= 2\left(- \frac{1}{8n} - \frac{1}{16n^2} -\frac{1}{32n^3} + O\Bigl(\frac{1}{n^4}\Bigr)\right),
= -\frac{1}{4n} + \frac{1}{8n^2} -\frac{1}{16n^3} + O\Bigl(\frac{1}{n^4}\Bigr).
\end{equation*}
\end{proposition}
\begin{proof}
When $\kappa = \frac{1}{2}$,
\begin{align*}
J_\infty(m_{(1^2)}) = I_n^{1/2}((1^2))  = I_n^{1/2}((0))\frac{n(n-1)}{4}\frac{n}{2n+1}
&= I_n^{1/2}((0))\,\frac{n(n-1)}{2}\, \Bigl(\frac{1}{4} - \frac{1}{8n} + \frac{1}{16n^2} -\frac{1}{32n^3} + O\Bigl(\frac{1}{n^4}\Bigr)\Bigr)
\\ &= I_n^{1/2}((0))\,n\,\Bigl(\frac{n}{8} - \frac{3}{16} + \frac{3}{32n} - \frac{3}{64n^2} + O\Bigl(\frac{1}{n^3}\Bigr)\Bigr)
\\ \hbox{and} \qquad I_n^{1/2}((2)) = I_n^{1/2}((0))\frac{n(n+2)}{12}\frac{n+3}{n+2} 
&= I_n^{1/2}((0))\,n\,\Bigl(\frac{n}{12} + \frac{1}{4} \Bigr).
\end{align*}
Therefore,
\begin{equation*}
J_\infty(m_{(2)}) = I_n^{1/2}((2)) - \frac{2}{3}I_n^{1/2}((1^2))
= I_n^{1/2}((0))\,n\,\Bigl(\frac{3}{8} - \frac{1}{16n} + \frac{1}{32n^2} + O\Bigl(\frac{1}{n^3}\Bigr)\Bigr),
\end{equation*}
which also gives
\begin{equation} \label{eqp:R-second-moment}
J_\infty\Bigl(\Bigl(t_1-\frac{1}{2}\Bigr)^2\Bigr) 
=  \frac{1}{n} J_\infty(m_{(2)}) -\frac{1}{n} J_\infty(m_{(1)}) + \frac{1}{4}J_\infty(1)
= I_n^{1/2}((0))\,\Bigl(\frac{1}{8} - \frac{1}{16n} + \frac{1}{32n^2} + O\Bigl(\frac{1}{n^3}\Bigr)\Bigr).
\end{equation}
Note also that 
\begin{equation} \label{eqp:R-linear-cross-term}
J_\infty\Bigl(\Bigl(t_1-\frac{1}{2}\Bigr)\Bigl(t_2-\frac{1}{2}\Bigr)\Bigr) 
= \frac{2}{n(n-1)} J_\infty(m_{(1^2)}) - \frac{1}{n} J_\infty(m_{(1)}) + \frac{1}{4}J_\infty(1)
= I_n^{1/2}((0))\,\Bigl(- \frac{1}{8n} + \frac{1}{16n^2} -\frac{1}{32n^3} + O\Bigl(\frac{1}{n^4}\Bigr)\Bigr).
\end{equation}

Next observe that
\begin{align*}
I_n^{1/2}((1^3))= I_n^{1/2}((0))\frac{n(n-1)(n-2)}{24}\frac{n-1}{2n+1}
&= I_n^{1/2}((0))\,\frac{n(n-1)}{2}\,\Bigl(\frac{n}{24} - \frac{7}{48} + \frac{5}{32n} - \frac{5}{64n^2}  + O\Bigl(\frac{1}{n^3}\Bigr)\Bigr)
\\ 
& = I_n^{1/2}((0))\,n\,\Bigl(\frac{n^2}{48} - \frac{3n}{32}+ \frac{29}{192} - \frac{15}{128n}  + O\Bigl(\frac{1}{n^2}\Bigr)\Bigr),
\end{align*}
\begin{align*}
I_n^{1/2}((2,1)) = I_n^{1/2}((0))\frac{n(n-1)(n+2)}{16}\frac{n+3}{n+2}\frac{n}{2n+1}
&= I_n^{1/2}((0))\,\frac{n(n-1)}{2}\,\Bigl(\frac{n}{16} + \frac{5}{32} - \frac{5}{64n} + \frac{5}{128n^2}  + O\Bigl(\frac{1}{n^3}\Bigr)\Bigr)
\\ 
& = I_n^{1/2}((0))\,n\,\Bigl(\frac{n^2}{32} + \frac{3n}{64} - \frac{15}{128} + \frac{15}{256n} + O\Bigl(\frac{1}{n^2}\Bigr)\Bigr)
\end{align*}
and 
\begin{equation*}
I_n^{1/2}((3)) = I_n^{1/2}((0))\frac{(n+4)(n+2)n}{120}\frac{n+5}{n+2}
= I_n^{1/2}((0))\,n\,\Bigl(\frac{n^2}{120} + \frac{3n}{40} + \frac{1}{6}\Bigr).
\end{equation*}
It follows that
\begin{equation*}
J_\infty(m_{(2,1)}) = I_n^{1/2}((2,1)) - \frac{3}{2}I_n^{1/2}((1^3))
= I_n^{1/2}((0))\,\frac{n(n-1)}{2}\,\Bigl(\frac{3}{8} - \frac{5}{16n} + \frac{5}{32n^2}  + O\Bigl(\frac{1}{n^3}\Bigr)\Bigr)
\end{equation*}
and
\begin{equation*} 
J_\infty(m_{(3)}) = I_n^{1/2}((3)) - \frac{3}{5}I_n^{1/2}((2,1)) + \frac{1}{2}I_n^{1/2}((1^3)) 
= I_n^{1/2}((0))\,n\,\Bigl(\frac{5}{16} - \frac{3}{32n} + O\Bigl(\frac{1}{n^2}\Bigr)\Bigr).
\end{equation*}

Moreover,
\begin{align*}
I_n^{1/2}((1^4)) = I_n^{1/2}((0))\frac{n(n-1)(n-2)(n-3)}{96}\frac{n-1}{2n+1}\frac{n-2}{2n-1}
&= I_n^{1/2}((0))\,\frac{n(n-1)}{2}\,\Bigl(\frac{n^2}{192} - \frac{n}{24} + \frac{31}{256} - \frac{5}{32n} + \frac{95}{1024n^2}  + O\Bigl(\frac{1}{n^3}\Bigr)\Bigr)
\\ 
& = I_n^{1/2}((0))\,n\,\Bigl(\frac{n^3}{384} - \frac{3n^2}{128}  + \frac{125n}{1536} - \frac{71}{512} + \frac{255}{2048n}  + O\Bigl(\frac{1}{n^2}\Bigr)\Bigr),
\end{align*}
\begin{align*}
I_n^{1/2}((2,1^2)) &= I_n^{1/2}((0))\frac{(n+2)n(n-1)(n-2)}{80}\frac{n+3}{n+2}\frac{n-1}{2n+1}
\\ &= I_n^{1/2}((0))\,\frac{n(n-1)}{2}\,\Bigl(\frac{n^2}{80} - \frac{n}{160} - \frac{27}{320} + \frac{15}{128n} - \frac{15}{256n^2}  + O\Bigl(\frac{1}{n^3}\Bigr)\Bigr)
\\ 
& = I_n^{1/2}((0))\,n\,\Bigl(\frac{n^3}{160} - \frac{3n^2}{320} - \frac{5n}{128} + \frac{129}{1280} - \frac{45}{512n} + O\Bigl(\frac{1}{n^2}\Bigr)\Bigr),
\end{align*}
\begin{align*}
I_n^{1/2}((2^2)) &= I_n^{1/2}((0))\frac{n(n-1)(n+2)(n+1)}{96} \frac{n+3}{n+2} \frac{n+2}{2n+3}\frac{n}{2n+1}
\\
&= I_n^{1/2}((0))\,\frac{n(n-1)}{2}\,\Bigl(\frac{n^2}{192} + \frac{n}{96} + \frac{3}{256} - \frac{1}{128n} + \frac{7}{1024n^2}  + O\Bigl(\frac{1}{n^3}\Bigr)\Bigr)
\\ 
& = I_n^{1/2}((0))\,n\,\Bigl(\frac{n^3}{384} + \frac{n^2}{128} - \frac{7n}{1536} - \frac{5}{512}
\frac{15}{2048n} + O\Bigl(\frac{1}{n^2}\Bigr)\Bigr),
\end{align*}
%%%
while
\begin{equation*}
I_n^{1/2}((3,1)) = I_n^{1/2}((0))\frac{(n+4)(n+2)n(n-1)}{144}\frac{n+5}{n+2}\frac{n}{2n+1}
= I_n^{1/2}((0))\,n\,\Bigl(\frac{n^3}{288} + \frac{5n^2}{192} + \frac{29n}{1152} - \frac{21}{256} + \frac{21}{512n} + O\Bigl(\frac{1}{n^2}\Bigr)\Bigr)
\end{equation*}
%%%%
and 
\begin{equation*}
I_n^{1/2}((4)) = I_n^{1/2}((0))\frac{(n+6)(n+4)(n+2)n}{1680}\frac{n+7}{n+4}\frac{n+5}{n+2}
= I_n^{1/2}((0))\,n\,\Bigl(\frac{n^3}{1680} + \frac{3n^2}{280} + \frac{107n}{1680} + \frac{1}{8}\Bigr).
\end{equation*}
It follows that
\begin{equation*}
J_\infty(m_{(2^2)}) = I_n^{1/2}((2^2)) - \frac{2}{3}I_n^{1/2}((2,1^2)) + \frac{3}{5}I_n^{1/2}((1^4))
= I_n^{1/2}((0))\,\frac{n(n-1)}{2}\,\Bigl(\frac{9}{64} - \frac{23}{128n} + \frac{13}{128n^2}  + O\Bigl(\frac{1}{n^3}\Bigr)\Bigr)
\end{equation*}
and 
\begin{equation*}
J_\infty(m_{(4)})  = I_n^{1/2}((4)) - \frac{4}{7}I_n^{1/2}((3,1)) - \frac{2}{15}I_n^{1/2}((2^2)) + \frac{4}{9}I_n^{1/2}((2,1^2)) -\frac{2}{5}I_n^{1/2}((1^4))
= I_n^{1/2}((0))\,n\,\Bigl(\frac{35}{128} - \frac{29}{256n} + O\Bigl(\frac{1}{n^2}\Bigr)\Bigr).
\end{equation*}

We conclude that
\begin{align} 
\nonumber
J_\infty\Bigl(\Bigl(t_1-\frac{1}{2}\Bigr)^2\Bigl(t_2-\frac{1}{2}\Bigr)^2\Bigr) 
&= \frac{2}{n(n-1)}J_\infty(m_{(2^2)}) - \frac{2}{n(n-1)}J_\infty(m_{(2,1)}) +\frac{1}{2n}J_\infty(m_{(2)})
\\ \nonumber & \hspace{3.2cm} +\frac{2}{n(n-1)}J_\infty(m_{(1^2)}) - \frac{1}{2n}J_\infty(m_{(1)}) +\frac{1}{16}J_\infty(1)
\\ \label{eqp:R-cross-term} &= I_n^{1/2}((0))\,\Bigl(\frac{1}{64} - \frac{3}{128n} + \frac{3}{128n^2} + O\Bigl(\frac{1}{n^3}\Bigr)\Bigr),
\end{align}
while
\begin{align} 
\nonumber
J_\infty\Bigl(\Bigl(t_1-\frac{1}{2}\Bigr)^4\Bigr) 
& =  \frac{1}{n} J_\infty(m_{(4)})- \frac{2}{n} J_\infty(m_{(3)}) + \frac{3}{2n} J_\infty(m_{(2)}) - \frac{1}{2n} J_\infty(m_{(1)})
+\frac{1}{16}J_\infty(1)
\\ \label{eqp:R-fourth-moment}
& = I_n^{1/2}((0))\,\Bigl(\frac{3}{128} - \frac{5}{256n} + O\Bigl(\frac{1}{n^2}\Bigr)\Bigr).
\end{align}
This completes the proof of Theorem \ref{thm:var-operator} when ${\mathbb F} = {\mathbb R}$.
\end{proof}

\begin{proposition}(Case of $\beta=4$, $\kappa=2$; ${\mathbb H}$-self-adjoint matrices)
The following estimates are true:
\begin{gather*}
\frac{N_\infty\bigl(x_1^2\bigr)}{N_\infty(1)} = 2\frac{J_\infty\Bigl(\Bigl(t_1-\frac{1}{2}\Bigr)^2\Bigr)}{J_\infty(1)}
%= 2\left(\frac{1}{8} + \frac{1}{32n} + \frac{1}{128n^2} + O\Bigl(\frac{1}{n^3}\Bigr)\right),
= \frac{1}{4} + \frac{1}{16n} + \frac{1}{64n^2} + O\Bigl(\frac{1}{n^3}\Bigr),
\\
\frac{N_\infty\bigl(x_1^2x_2^2\bigr)}{N_\infty(1)} = 4\frac{J_\infty\Bigl(\Bigl(t_1-\frac{1}{2}\Bigr)^2\Bigl(t_2-\frac{1}{2}\Bigr)^2\Bigr)}{J_\infty(1)}
%= 4\left(\frac{1}{64}  - \frac{3}{1024n^2} + O\Bigl(\frac{1}{n^3}\Bigr)\right)
= \frac{1}{16}  - \frac{3}{256n^2} + O\Bigl(\frac{1}{n^3}\Bigr)
\\
\intertext{and}
\frac{N_\infty\bigl(x_1^4\bigr)}{N_\infty(1)} = 4\frac{J_\infty\Bigl(\Bigl(t_1-\frac{1}{2}\Bigr)^4\Bigr)}{J_\infty(1)}
%= 4\left(\frac{3}{128} + \frac{5}{512n} + O\Bigl(\frac{1}{n^2}\Bigr)\right).
= \frac{3}{32} + \frac{5}{128n} + O\Bigl(\frac{1}{n^2}\Bigr).
\end{gather*}
As a consequence,
\begin{equation*}
{\rm Var}_{N_\infty}\bigl(\|x\|_2^2\bigr) 
%= n\frac{N_\infty\bigl(x_1^4\bigr)}{N_\infty(1)}
%+n(n-1)\frac{N_\infty\bigl(x_1^2x_2^2\bigr)}{N_\infty(1)} 
%- n^2 \left(\frac{N_\infty\bigl(x_1^2\bigr)}{N_\infty(1)}\right)^2 
= \frac{1}{64} + O\Bigl(\frac{1}{n}\Bigr).
\end{equation*}
\end{proposition}
\begin{proof}
When $\kappa = 2$,
\begin{align*}
J_\infty(m_{(1^2)}) = I_n^{2}((1^2)) =
I_n^2((0))\frac{n(n-1)}{16}\frac{2n-3}{n-1}
&= I_n^2((0))\,\frac{n(n-1)}{2}\, \Bigl(\frac{1}{4} - \frac{1}{8n} - \frac{1}{8n^2} -\frac{1}{8n^3} + O\Bigl(\frac{1}{n^4}\Bigr)\Bigr)
\\ &= I_n^2((0))\,n\,\Bigl(\frac{n}{8} - \frac{3}{16}\Bigr)
\\ \hbox{and} \qquad I_n^{2}((2))  = I_n^2((0))\frac{n(2n+1)}{6}\frac{2n}{4n-1} 
&= I_n^2((0))\,n\,\Bigl(\frac{n}{6} + \frac{1}{8} + \frac{1}{32n} + \frac{1}{128n^2} + O\Bigl(\frac{1}{n^3}\Bigr)\Bigr). \end{align*}
Therefore,
\begin{equation*}
J_\infty(m_{(2)}) = I_n^{2}((2)) - \frac{4}{3}I_n^{2}((1^2))
= I_n^2((0))\,n\,\Bigl(\frac{3}{8} + \frac{1}{32n} + \frac{1}{128n^2} + O\Bigl(\frac{1}{n^3}\Bigr)\Bigr),
\end{equation*}
which also gives
\begin{equation} \label{eqp:H-second-moment}
J_\infty\Bigl(\Bigl(t_1-\frac{1}{2}\Bigr)^2\Bigr) 
=  \frac{1}{n} J_\infty(m_{(2)}) -\frac{1}{n} J_\infty(m_{(1)}) + \frac{1}{4}J_\infty(1)
= I_n^2((0))\,\Bigl(\frac{1}{8} + \frac{1}{32n} + \frac{1}{128n^2} + O\Bigl(\frac{1}{n^3}\Bigr)\Bigr).
\end{equation}

Next observe that
\begin{align*}
I_n^{2}((1^3)) = I_n^2((0))\frac{n(n-1)(n-2)}{96}\frac{2n-5}{n-1}
&= I_n^2((0))\,\frac{n(n-1)}{2}\,\Bigl(\frac{n}{24} - \frac{7}{48} + \frac{1}{16n} + \frac{1}{16n^2}  + O\Bigl(\frac{1}{n^3}\Bigr)\Bigr)
\\ 
& = I_n^2((0))\,n\,\Bigl(\frac{n^2}{48} - \frac{3n}{32}+ \frac{5}{48} + O\Bigl(\frac{1}{n^2}\Bigr)\Bigr),
\end{align*}
\begin{align*}
I_n^2((2,1)) = I_n^2((0))\frac{n(n-1)(2n+1)}{20}\frac{n}{4n-1}\frac{2n-3}{n-1}
&= I_n^2((0))\,\frac{n(n-1)}{2}\,\Bigl(\frac{n}{10} + \frac{1}{40} - \frac{11}{160n} - \frac{59}{640n^2}  + O\Bigl(\frac{1}{n^3}\Bigr)\Bigr)
\\ 
& = I_n^2((0))\,n\,\Bigl(\frac{n^2}{20} - \frac{3n}{80} - \frac{3}{64} - \frac{3}{256n} + O\Bigl(\frac{1}{n^2}\Bigr)\Bigr)
\end{align*}
and 
\begin{equation*}
I_n^2((3)) = I_n^2((0))\frac{(n+1)(2n+1)n}{24}\frac{2n+1}{4n-1}
= I_n^2((0))\,n\,\Bigl(\frac{n^2}{24} + \frac{3n}{32} + \frac{29}{384} + \frac{15}{512n} + O\Bigl(\frac{1}{n^2}\Bigr)\Bigr).
\end{equation*}
It follows that
\begin{equation*}
J_\infty(m_{(2,1)}) = I_n^2((2,1)) - \frac{12}{5}I_n^2((1^3))
= I_n^2((0))\,\frac{n(n-1)}{2}\,\Bigl(\frac{3}{8} - \frac{7}{32n} - \frac{31}{128n^2}  + O\Bigl(\frac{1}{n^3}\Bigr)\Bigr)
\end{equation*}
and
\begin{equation*} 
J_\infty(m_{(3)}) = I_n^2((3)) - \frac{3}{2}I_n^2((2,1)) + \frac{8}{5}I_n^2((1^3)) 
= I_n^2((0))\,n\,\Bigl(\frac{5}{16} + \frac{3}{64n} + O\Bigl(\frac{1}{n^2}\Bigr)\Bigr).
\end{equation*}

Moreover,
\begin{align*}
I_n^2((1^4)) = I_n^2((0))\frac{n(n-1)(n-2)(n-3)}{1536}\frac{2n-5}{n-1}\frac{2n-7}{n-2}
&= I_n^2((0))\,\frac{n(n-1)}{2}\,\Bigl(\frac{n^2}{192} - \frac{n}{24} + \frac{25}{256} - \frac{5}{128n} - \frac{5}{128n^2}  + O\Bigl(\frac{1}{n^3}\Bigr)\Bigr)
\\ 
& = I_n^2((0))\,n\,\Bigl(\frac{n^3}{384} - \frac{3n^2}{128}  + \frac{107n}{1536} - \frac{35}{1536} + O\Bigl(\frac{1}{n^2}\Bigr)\Bigr),
\end{align*}
\begin{align*}
I_n^2((2,1^2)) &= I_n^2((0))\frac{(2n+1)n(n-1)(n-2)}{112}\frac{n}{4n-1}\frac{2n-5}{n-1}
\\
&= I_n^2((0))\,\frac{n(n-1)}{2}\,\Bigl(\frac{n^2}{56} - \frac{11n}{224} - \frac{15}{896} + \frac{129}{3584n} + \frac{705}{14336n^2}  + O\Bigl(\frac{1}{n^3}\Bigr)\Bigr)
\\ 
& = I_n^2((0))\,n\,\Bigl(\frac{n^3}{112} - \frac{15n^2}{448} + \frac{29n}{1792} + \frac{27}{1024} + \frac{27}{4096n} + O\Bigl(\frac{1}{n^2}\Bigr)\Bigr),
\end{align*}
\begin{align*}
I_n^2((2^2)) &= I_n^2((0))\frac{n(n-1)(2n+1)(2n-1)}{60}\frac{2n}{4n-1}\frac{2n-2}{4n-3}\frac{2n-3}{4n-4}
\\
&= I_n^2((0))\,\frac{n(n-1)}{2}\,\Bigl(\frac{n^2}{60} - \frac{n}{120} - \frac{1}{64} - \frac{1}{128n} - \frac{5}{1024n^2}  + O\Bigl(\frac{1}{n^3}\Bigr)\Bigr)
\\ 
& = I_n^2((0))\,n\,\Bigl(\frac{n^3}{120} - \frac{n^2}{80} - \frac{7n}{1920} + \frac{1}{256} + \frac{3}{2048n} + O\Bigl(\frac{1}{n^2}\Bigr)\Bigr),
\end{align*}
%%%
while
\begin{equation*}
I_n^2((3,1)) = I_n^2((0))\frac{(2n+2)(2n+1)n(n-1)}{72}\frac{2n+1}{4n-1}\frac{2n-3}{4n-4}
= I_n^2((0))\,n\,\Bigl(\frac{n^3}{72} + \frac{n^2}{96} - \frac{25n}{1152} - \frac{43}{1536} - \frac{25}{2048n} + O\Bigl(\frac{1}{n^2}\Bigr)\Bigr)
\end{equation*}
%%%%
and 
\begin{equation*}
I_n^2((4)) = I_n^2((0))\frac{(2n+3)(2n+2)(2n+1)n}{240}\frac{2n+2}{4n+1}\frac{2n+1}{4n-1}
= I_n^2((0))\,n\,\Bigl(\frac{n^3}{120} + \frac{3n^2}{80} + \frac{25n}{384} + \frac{71}{1280} + \frac{51}{2048n} + O\Bigl(\frac{1}{n^2}\Bigr)\Bigr).
\end{equation*}
It follows that
\begin{equation*}
J_\infty(m_{(2^2)}) = I_n^2((2^2)) - \frac{4}{3}I_n^2((2,1^2)) + \frac{48}{35}I_n^2((1^4))
= I_n^2((0))\,\frac{n(n-1)}{2}\,\Bigl(\frac{9}{64} - \frac{7}{64n} - \frac{127}{1024n^2}  + O\Bigl(\frac{1}{n^3}\Bigr)\Bigr)
\end{equation*}
and 
\begin{equation*}
J_\infty(m_{(4)})  = I_n^2((4)) - \frac{8}{5}I_n^2((3,1)) + \frac{1}{3}I_n^2((2^2)) + \frac{16}{9}I_n^2((2,1^2)) - \frac{64}{35}I_n^2((1^4))
= I_n^2((0))\,n\,\Bigl(\frac{35}{128} + \frac{29}{512n} + O\Bigl(\frac{1}{n^2}\Bigr)\Bigr).
\end{equation*}

We conclude that
\begin{align} 
\nonumber
J_\infty\Bigl(\Bigl(t_1-\frac{1}{2}\Bigr)^2\Bigl(t_2-\frac{1}{2}\Bigr)^2\Bigr) 
&= \frac{2}{n(n-1)}J_\infty(m_{(2^2)}) - \frac{2}{n(n-1)}J_\infty(m_{(2,1)}) +\frac{1}{2n}J_\infty(m_{(2)})
\\ \nonumber & \hspace{3.2cm} +\frac{2}{n(n-1)}J_\infty(m_{(1^2)}) - \frac{1}{2n}J_\infty(m_{(1)}) +\frac{1}{16}J_\infty(1)
\\ \label{eqp:H-cross-term} &= I_n^2((0))\,\Bigl(\frac{1}{64} - \frac{3}{1024n^2} + O\Bigl(\frac{1}{n^3}\Bigr)\Bigr),
\end{align}
while
\begin{align} 
\nonumber
J_\infty\Bigl(\Bigl(t_1-\frac{1}{2}\Bigr)^4\Bigr) 
& =  \frac{1}{n} J_\infty(m_{(4)})- \frac{2}{n} J_\infty(m_{(3)}) + \frac{3}{2n} J_\infty(m_{(2)}) - \frac{1}{2n} J_\infty(m_{(1)})
+\frac{1}{16}J_\infty(1)
\\ \label{eqp:H-fourth-moment}
& = I_n^2((0))\,\Bigl(\frac{3}{128} + \frac{5}{512n}  + O\Bigl(\frac{1}{n^2}\Bigr)\Bigr).
\end{align}
%%%%%%%%%
This completes the proof of Theorem \ref{thm:var-operator} in all cases.
\end{proof}

\begin{remark}\label{rem:beta-ensembles-estimates}
We can sum up the above computations, which can be made for all large enough $\beta$,
as follows: as long as $\beta = 2\kappa$ is bounded away from zero, 
i.e. $\beta \gs \beta_0$ for some fixed $\beta_0 > 0$,
we have
\begin{align*} 
& \frac{1}{I_n^{\beta/2}((0))} \cdot \left(\int_{[-\frac{1}{2},\frac{1}{2}]^n} m_{(4)}({\bm x})\ |\Delta_n({\bm x})|^\beta\,d{\bm x} 
\ + \ 2 \int_{[-\frac{1}{2},\frac{1}{2}]^n} m_{(2^2)}({\bm x})\,|\Delta_n({\bm x})|^\beta\,d{\bm x} \right)
\\
& \hspace{8cm} 
- \left(\frac{1}{I_n^{\beta/2}((0))}\int_{[-\frac{1}{2},\frac{1}{2}]^n} m_{(2)}({\bm x})\ |\Delta_n({\bm x})|^\beta\,d{\bm x} \right)^2
\\
= &\ \left(\frac{3}{128}n + \frac{5(\beta-2)}{256\beta}\right) +  \left(\frac{1}{64}n^2 - \frac{\beta+4}{128\beta}n + \frac{\beta^2-9\beta+14}{128\beta^2}\right)
- \left(\frac{1}{64}n^2 + \frac{\beta-2}{64\beta}n + \frac{7\beta^2 - 32\beta + 28}{256\beta^2}\right) + O_{\beta_0}\left(\frac{1}{n}\right) 
\\
= &\ \frac{1}{64\beta} + O_{\beta_0}\left(\frac{1}{n}\right). 
\end{align*}
\end{remark}

\section{Almost isotropicity of $B_E$ in the subspaces of self-adjoint matrices}\label{sec:almost-isotropic}

Here we establish Theorem \ref{thm:almost-isotropic}. 

\medskip

\noindent\emph{Proof in the case where $E$ is the subspace of Hermitian matrices.}
The orthonormal basis that we fix is the following: 
\begin{equation*} \{J^{kk}: 1\ls k\ls n\}\bigcup \bigl\{\tfrac{1}{\sqrt{2}}\bigl(J^{kl} + J^{lk}\bigr): k < l\bigr\}
\bigcup \bigl\{\tfrac{i}{\sqrt{2}}\bigl(J^{kl} - J^{lk}\bigr): k < l\bigr\} \end{equation*}
where $J^{kl}$ 
is the single-entry matrix whose only non-zero entry is the $(k,l)$-th one and is equal to 1.
According to Theorem \ref{thm:CMS-C-conj-inv}, we have
\begin{equation*}
\frac{1}{{\rm vol}(B_E)}\int_{B_E} T_{k_1l_1}T_{k_2l_2}\,dT = 0
\end{equation*}
whenever $\{k_1,k_2\} \neq \{l_1,l_2\}$.
This immediately shows that any pair of marginals of the distribution which correspond
to one diagonal and (either the real or the imaginary part of) one non-diagonal entry 
is linearly uncorrelated. Similarly, if they correspond to two non-diagonal entries $(k_1,l_1)$,
$(k_2,l_2)$ with $(k_2,l_2)\notin \{(k_1,l_1), (l_1,k_1)\}$ we can observe the following:
\begin{align*}
0 & = \frac{1}{{\rm vol}(B_E)}\int_{B_E} T_{k_1l_1}T_{k_2l_2}\,dT
\\
& = \frac{1}{{\rm vol}(B_E)}\int_{B_E} \bigl(\Re(T_{k_1l_1})\Re(T_{k_2l_2}) - \Im(T_{k_1l_1})\Im(T_{k_2l_2})\bigr)\,dT
\\
& \qquad 
+ \frac{i}{{\rm vol}(B_E)}\int_{B_E} \bigl(\Re(T_{k_1l_1})\Im(T_{k_2l_2}) + \Im(T_{k_1l_1})\Re(T_{k_2l_2})\bigr)\,dT,
\end{align*}
while 
\begin{align*}
0 & = \frac{1}{{\rm vol}(B_E)}\int_{B_E} T_{k_1l_1}T_{l_2k_2}\,dT
\\
& = \frac{1}{{\rm vol}(B_E)}\int_{B_E} \bigl(\Re(T_{k_1l_1})\Re(T_{l_2k_2}) - \Im(T_{k_1l_1})\Im(T_{l_2k_2})\bigr)\,dT
\\
& \qquad 
+ \frac{i}{{\rm vol}(B_E)}\int_{B_E} \bigl(\Re(T_{k_1l_1})\Im(T_{l_2k_2}) + \Im(T_{k_1l_1})\Re(T_{l_2k_2})\bigr)\,dT
\\
& = \frac{1}{{\rm vol}(B_E)}\int_{B_E} \bigl(\Re(T_{k_1l_1})\Re(T_{k_2l_2}) + \Im(T_{k_1l_1})\Im(T_{k_1l_2})\bigr)\,dT
\\
& \qquad + \frac{i}{{\rm vol}(B_E)}\int_{B_E} \bigl(-\Re(T_{k_1l_1})\Im(T_{k_2l_2}) + \Im(T_{k_1l_1})\Re(T_{k_2l_2})\bigr)\,dT.
\end{align*}
Combined, these show that all the above integrals are equal to 0.

Let us examine the remaining cases, where the marginals correspond to two different diagonal entries
$(k,k), (l,l)$, or to the real and to the imaginary part of the same non-diagonal entry $(k,l)$, $k\neq l$.

In the latter case, we can write
\begin{align}\label{eqp1:C-almost-isotropicity}
0 & = \frac{1}{{\rm vol}(B_E)}\int_{B_E} T_{kl}T_{kl}\,dT
\\ \nonumber
& = \frac{1}{{\rm vol}(B_E)}\int_{B_E} \bigl(\Re(T_{kl})^2 - \Im(T_{kl})^2\bigr)\,dT
+ \frac{2i}{{\rm vol}(B_E)}\int_{B_E} \Re(T_{kl})\Im(T_{kl})\,dT,
\end{align}
which shows that the marginals are uncorrelated.

In the former case, we have from Theorem \ref{thm:CMS-C-conj-inv} 
and from Proposition \ref{prop:C-trace-estimates} that
\begin{align*}
\frac{1}{{\rm vol}(B_E)}\int_{B_E} T_{kk}T_{ll}\,dT 
& = {\rm Wg}^U(e;n)\ \frac{1}{{\rm vol}(B_E)}\int_{B_E}{\rm Tr}_e(T) \,dT 
+ {\rm Wg}^U((12);n)\frac{1}{{\rm vol}(B_E)}\int_{B_E}{\rm Tr}_{(12)}(T) \,dT
\\
& = \frac{1}{(n-1)(n+1)}\,\frac{1}{{\rm vol}(B_E)}\int_{B_E}\bigl({\rm Tr}(T)\bigr)^2 \,dT
- \frac{1}{n(n-1)(n+1)}\,\frac{1}{{\rm vol}(B_E)}\int_{B_E}{\rm Tr}(T^2) \,dT
\\
& = \frac{1}{(n-1)(n+1)}\,\left(n\frac{N_\infty\bigl(x_1^2\bigr)}{N_\infty(1)} + n(n-1)\frac{N_\infty\bigl(x_1x_2\bigr)}{N_\infty(1)} \right) - \frac{1}{(n-1)(n+1)}\,\frac{N_\infty\bigl(x_1^2\bigr)}{N_\infty(1)} 
\\
& = \frac{1}{n+1} \left(\frac{N_\infty\bigl(x_1^2\bigr)}{N_\infty(1)} + n\frac{N_\infty\bigl(x_1x_2\bigr)}{N_\infty(1)} \right) 
\\
& = - \frac{1}{8n(n+1)} + O\Bigl(\frac{1}{n^3}\Bigr).
\end{align*}

\bigskip

Moreover, turning to second moments of the marginals, we see that
\begin{align*}
\frac{1}{{\rm vol}(B_E)}\int_{B_E} T_{kk}^2\,dT 
& = {\rm Wg}^U(e;n)\ \frac{1}{{\rm vol}(B_E)}\int_{B_E}{\rm Tr}_e(T) \,dT 
+ {\rm Wg}^U((12);n)\frac{1}{{\rm vol}(B_E)}\int_{B_E}{\rm Tr}_{(12)}(T) \,dT
\\
&\qquad  + {\rm Wg}^U((12);n)\ \frac{1}{{\rm vol}(B_E)}\int_{B_E}{\rm Tr}_e(T) \,dT 
+ {\rm Wg}^U(e;n)\frac{1}{{\rm vol}(B_E)}\int_{B_E}{\rm Tr}_{(12)}(T) \,dT
\\
& = - \frac{1}{8n(n+1)} + O\Bigl(\frac{1}{n^3}\Bigr)
\\
& \qquad - \frac{1}{n(n-1)(n+1)}\,\frac{1}{{\rm vol}(B_E)}\int_{B_E}\bigl({\rm Tr}(T)\bigr)^2 \,dT
+ \frac{1}{(n-1)(n+1)}\,\frac{1}{{\rm vol}(B_E)}\int_{B_E}{\rm Tr}(T^2) \,dT
\\
& = - \frac{1}{8n(n+1)}  + O\Bigl(\frac{1}{n^3}\Bigr)
\\
& \qquad - \frac{1}{n(n-1)(n+1)}\,
\left(n\frac{N_\infty\bigl(x_1^2\bigr)}{N_\infty(1)} + n(n-1)\frac{N_\infty\bigl(x_1x_2\bigr)}{N_\infty(1)} \right)
+ \frac{n}{(n-1)(n+1)}\,\frac{N_\infty\bigl(x_1^2\bigr)}{N_\infty(1)}
\\
& = - \frac{1}{8n(n+1)}  + O\Bigl(\frac{1}{n^3}\Bigr)
\\
& \qquad - \frac{1}{n(n-1)(n+1)}\,
\left(\frac{1}{8} + O\Bigl(\frac{1}{n^2}\Bigr)\right)
+ \frac{n}{(n-1)(n+1)}\,\left(\frac{1}{4} - \frac{1}{16n^2} + O\Bigl(\frac{1}{n^3}\Bigr)\right)
\\
& = \frac{n}{4(n-1)(n+1)} + O\Bigl(\frac{1}{n^2}\Bigr).
\end{align*}

On the other hand, when we consider a non-diagonal entry $(k,l)$, \eqref{eqp1:C-almost-isotropicity} shows that
\begin{equation*} \frac{1}{{\rm vol}(B_E)}\int_{B_E} \Re(T_{kl})^2\,dT
=  \frac{1}{{\rm vol}(B_E)}\int_{B_E}\Im(T_{kl})^2\,dT. \end{equation*}
To compute this integral, we note that
\begin{align*}
\frac{1}{{\rm vol}(B_E)}\int_{B_E} 2\Re(T_{kl})^2\,dT 
&= \frac{1}{{\rm vol}(B_E)}\int_{B_E} \bigl(\Re(T_{kl})^2 + \Im(T_{kl})^2\bigr)\,dT
= \frac{1}{{\rm vol}(B_E)}\int_{B_E} T_{kl}T_{lk}\,dT
\\
& = {\rm Wg}^U((12);n)\ \frac{1}{{\rm vol}(B_E)}\int_{B_E}{\rm Tr}_e(T) \,dT 
+ {\rm Wg}^U(e;n)\frac{1}{{\rm vol}(B_E)}\int_{B_E}{\rm Tr}_{(12)}(T) \,dT
\\
&= - \frac{1}{n(n-1)(n+1)}\,
\left(\frac{1}{8} + O\Bigl(\frac{1}{n^2}\Bigr)\right)
+ \frac{n}{(n-1)(n+1)}\,\left(\frac{1}{4} - \frac{1}{16n^2} + O\Bigl(\frac{1}{n^3}\Bigr)\right)
\\
&= \frac{n}{4(n-1)(n+1)} + O\Bigl(\frac{1}{n^3}\Bigr).
\end{align*}

We conclude that the covariance matrix ${\rm Cov}(B_E)$ of $B_E$ has the following form:
all its diagonal entries are $= \frac{n}{4(n-1)(n+1)} + O\Bigl(\frac{1}{n^2}\Bigr)$,
while the only non-zero non-diagonal entries are those giving the correlation between
marginals corresponding to two different diagonal entries of $T\in B_E$,
and these are $= - \frac{1}{8n(n+1)} + O\Bigl(\frac{1}{n^3}\Bigr)$. It follows that,
in order to find all eigenvalues of ${\rm Cov}(B_E)$, it suffices to find the eigenvalues
of the $n\times n$ submatrix $D_{B_E}$ which involves only the marginals corresponding to
diagonal entries of $T\in B_E$ (since the remaining eigenvalues are all $= \frac{n}{4(n-1)(n+1)} + O\Bigl(\frac{1}{n^2}\Bigr)$ as immediately seen from the form of ${\rm Cov}(B_E)$).

The submatrix $D_{B_E}$ is of the form
\begin{equation*} (a-b)I_n + bJ_n \end{equation*}
where $J_n$ is the matrix with all entries equal to 1 and $a = \frac{n}{4(n-1)(n+1)} + O\Bigl(\frac{1}{n^2}\Bigr)$,
$b = - \frac{1}{8n(n+1)} + O\Bigl(\frac{1}{n^3}\Bigr)$.
It is not difficult to see that such a matrix can only have two eigenvalues: 
the eigenvalue $a+(n-1)b$ (corresponding to the vector $(1,1,\ldots,1)$)
and the eigenvalue $a-b$ (which will have mutliplicity $n-1$). In our case,
these eigenvalues are $= \frac{1}{8(n+1)} + O\Bigl(\frac{1}{n^2}\Bigr)$
and $= \frac{n}{4(n-1)(n+1)} + O\Bigl(\frac{1}{n^2}\Bigr)$ respectively.
This shows that all eigenvalues of $D_{B_E}$, and thus of ${\rm Cov}(B_E)$ too, are approximately equal.

Finally, the covariance matrix ${\rm Cov}(\overline{B_E})$ of the volume-normalised unit ball $\overline{B_E}$ 
can be found by multiplying ${\rm Cov}(B_E)$ by $[{\rm vol}(B_E)]^{-2/n^2} \simeq n$.
\qed

\bigskip
\bigskip

\noindent\emph{Proof in the case where $E$ is the subspace of $\ {\mathbb R}$-self-adjoint matrices.}
Our aim is to compute integrals of the form 
\begin{equation*}
 \frac{1}{{\rm vol}(B_E)}\int_{B_E} T_{j_1l_1}T_{j_2l_2}\,dT, \qquad \qquad 1\ls j_1, j_2, l_1, l_2\ls n, \end{equation*}
so we apply Theorem \ref{thm:CMS-R-conj-inv} with $k=2$. Here 
\begin{align*}
M_{2k} = M_4 &= \Bigl\{\bigl\{\{1,2\}, \{3,4\}\bigr\}, \bigl\{\{1,3\}, \{2,4\}\bigr\}, \bigl\{\{1,4\}, \{2,3\}\bigr\}\Bigr\}
\\
\intertext{and if we express the pair partitions as permutations in $S_{2k}$ per our convention}
& = \bigl\{e, (23), (243)\bigr\}.
\end{align*}
Therefore, 
\begin{equation*}
M_4^{-1}M_4 := \bigl\{\sigma^{-1}\tau: \sigma, \tau\in M_4\bigr\}
= \bigl\{e, (23), (24), (243), (234)\bigr\},
\end{equation*}
and all these permutations have coset-type $(2)$ except for the trivial permutation $e$ which has coset-type $(1^2)$.

Moreover,
\begin{equation*}
H_2 = \bigl\langle (12), (34), (13)(24) \bigr\rangle 
= \bigl\{e, (12), (34), (13)(24), (12)(34), (14)(23), (1324), (1423)\bigr\}.
\end{equation*}
To compute the orthogonal Weingarten function on $S_4$, we first find the 
zonal spherical functions $\omega^{(2)}$ and $\omega^{(1^2)}$. It is easily seen that
\begin{equation*}
\omega^{(2)}(\sigma) = \frac{1}{8}\sum_{\zeta\in H_2}\chi^{(4)}(\sigma\zeta) = 1
\qquad \hbox{for every $\sigma \in S_4$.}
\end{equation*} 
On the other hand,
\begin{equation*}
\omega^{(1^2)}(e) = \frac{1}{8}\sum_{\zeta\in H_2}\chi^{(2^2)}(\zeta) = 1,
\end{equation*} 
while
\begin{equation*}
\omega^{(1^2)}(\sigma) = \omega^{(1^2)}((23))= \frac{1}{8}\sum_{\zeta\in H_2}\chi^{(2^2)}\bigl((23)\zeta\bigr) = -\frac{1}{2} \qquad \hbox{for every $\sigma \in S_4$ with coset-type $(2)$}
\end{equation*} 
(in particular for every permutation $\sigma \in M_4^{-1}M_4\setminus \{e\}$).

We can now compute:
\begin{align*}
{\rm Wg}^O(\sigma;n) &= \frac{8}{24}\sum_{\lambda\vdash 2}\frac{\chi^{2\lambda}(e)}{C_\lambda^\prime(n)}\,\omega^\lambda(\sigma) 
\\
& = \frac{1}{3}\left(\frac{\chi^{(4)}(e)\omega^{(2)}(\sigma)}{C_{(2)}^\prime(n)} 
+ \frac{\chi^{(2^2)}(e)\omega^{(1^2)}(\sigma)}{C_{(1^2)}^\prime(n)}\right)
 = \left\{\begin{array}{cl} \frac{n+1}{n(n-1)(n+2)} &  \hbox{if $\sigma = e$}
\\ \\
 -\frac{1}{n(n-1)(n+2)} &  \hbox{if $\sigma \in M_4^{-1}M_4\setminus \{e\}$}
\end{array}\right..
\end{align*}

\bigskip

The orthonormal basis that we have fixed is the following: 
\begin{equation*} \{J^{kk}: 1\ls k\ls n\}\bigcup \bigl\{\tfrac{1}{\sqrt{2}}\bigl(J^{kl} + J^{lk}\bigr): k < l\bigr\}.
\end{equation*}
According to Theorem \ref{thm:CMS-R-conj-inv}, we have
\begin{equation*}
 \frac{1}{{\rm vol}(B_E)}\int_{B_E} T_{i_1i_2}T_{i_3i_4}\,dT = 0
\end{equation*}
if there is at least one index that appears an odd number of times among the $i_j, j=1,\ldots,4$.
This immediately shows that marginals of the distribution which correspond
to two different non-diagonal entries or to one non-diagonal and one diagonal entry
are linearly uncorrelated. 

The only other case, where we have correlation, is when $i_1=i_2=j \neq k = i_3 = i_4$.
In this case
\begin{align*}
 \frac{1}{{\rm vol}(B_E)}\int_{B_E} T_{jj}T_{kk}\,dT
 &= {\rm Wg}^O(e;n)\ \frac{1}{{\rm vol}(B_E)}\int_{B_E}{\rm Tr}_e^\prime(T)\,dT
 + {\rm Wg}^O((23);n)\ \frac{1}{{\rm vol}(B_E)}\int_{B_E}{\rm Tr}_{(23)}^\prime(T)\,dT
 \\
 &\qquad + {\rm Wg}^O((243);n)\ \frac{1}{{\rm vol}(B_E)}\int_{B_E}{\rm Tr}_{(243)}^\prime(T)\,dT
 \\
 & = \frac{n+1}{n(n-1)(n+2)} \ \frac{1}{{\rm vol}(B_E)}\int_{B_E}\bigl({\rm Tr}(T)\bigr)^2\,dT
 -\frac{2}{n(n-1)(n+2)} \ \frac{1}{{\rm vol}(B_E)}\int_{B_E}{\rm Tr}(T^2)\,dT
 \\
 & = \frac{n+1}{n(n-1)(n+2)}\ \left(n\frac{N_\infty\bigl(x_1^2\bigr)}{N_\infty(1)} + n(n-1)\frac{N_\infty\bigl(x_1x_2\bigr)}{N_\infty(1)} \right)
  -\frac{2}{(n-1)(n+2)}\ \frac{N_\infty\bigl(x_1^2\bigr)}{N_\infty(1)}
 \\
 & = \frac{1}{n+2}\ \frac{N_\infty\bigl(x_1^2\bigr)}{N_\infty(1)} + \frac{n+1}{n+2}\ \frac{N_\infty\bigl(x_1x_2\bigr)}{N_\infty(1)}
 \\ 
 & = -\frac{1}{4n(n+2)} + O\Bigl(\frac{1}{n^3}\Bigr).
\end{align*}

Turning to second moments, we first handle the case $i_1=i_3=j \neq k = i_2 = i_4$:
\begin{align*}
\frac{1}{{\rm vol}(B_E)}\int_{B_E} \bigl(\tfrac{1}{\sqrt{2}}\bigl(T_{jk} + T_{kj}\bigr)\bigr)^2\,dT
&= \frac{1}{{\rm vol}(B_E)}\int_{B_E} 2 T_{jk}^2\,dT
\\
& = 2\left({\rm Wg}^O((23);n)\ \frac{1}{{\rm vol}(B_E)}\int_{B_E}{\rm Tr}_e^\prime(T)\,dT
 + {\rm Wg}^O(e;n)\ \frac{1}{{\rm vol}(B_E)}\int_{B_E}{\rm Tr}_{(23)}^\prime(T)\,dT\right.
 \\
 & \qquad \qquad \left.+ \, {\rm Wg}^O((24);n)\ \frac{1}{{\rm vol}(B_E)}\int_{B_E}{\rm Tr}_{(243)}^\prime(T)\,dT\right)
 \\
 & = 2\left(-\frac{1}{n(n-1)(n+2)} \ \frac{1}{{\rm vol}(B_E)}\int_{B_E}\bigl({\rm Tr}(T)\bigr)^2\,dT
 +  \frac{n+1}{n(n-1)(n+2)}\ \frac{1}{{\rm vol}(B_E)}\int_{B_E}{\rm Tr}(T^2)\,dT\right.
 \\
 & \qquad \qquad \left. -\, \frac{1}{n(n-1)(n+2)} \ \frac{1}{{\rm vol}(B_E)}\int_{B_E}{\rm Tr}(T^2)\,dT \right)
 \\
 & = -\frac{2}{n(n-1)(n+2)}\ \left(n\frac{N_\infty\bigl(x_1^2\bigr)}{N_\infty(1)} + n(n-1)\frac{N_\infty\bigl(x_1x_2\bigr)}{N_\infty(1)} \right) + \frac{2}{(n-1)(n+2)} \ n\frac{N_\infty\bigl(x_1^2\bigr)}{N_\infty(1)} 
 \\
 & = \frac{2}{n+2}\ \left(\frac{N_\infty\bigl(x_1^2\bigr)}{N_\infty(1)} - \frac{N_\infty\bigl(x_1x_2\bigr)}{N_\infty(1)}\right)
 \\
 & = \frac{1}{2(n+2)} + O\Bigl(\frac{1}{n^2}\Bigr).
\end{align*}

Finally,
\begin{align*}
\frac{1}{{\rm vol}(B_E)}\int_{B_E} T_{jj}^2\, dT
& = \sum_{\sigma\in M_4} \left({\rm Wg}^O(\sigma^{-1}e;n)\ \frac{1}{{\rm vol}(B_E)}\int_{B_E}{\rm Tr}_e^\prime(T)\,dT
 + {\rm Wg}^O(\sigma^{-1}(23);n)\ \frac{1}{{\rm vol}(B_E)}\int_{B_E}{\rm Tr}_{(23)}^\prime(T)\,dT\right.
 \\
 & \qquad\qquad \qquad \left.+ \, {\rm Wg}^O(\sigma^{-1}(243);n)\ \frac{1}{{\rm vol}(B_E)}\int_{B_E}{\rm Tr}_{(243)}^\prime(T)\,dT\right)
\\
& = \frac{1}{{\rm vol}(B_E)}\int_{B_E} T_{jj}T_{kk}\, dT + 2\cdot \frac{1}{{\rm vol}(B_E)}\int_{B_E} T_{jk}^2\, dT
\\
& = \frac{1}{2(n+2)} + O\Bigl(\frac{1}{n^2}\Bigr).
\end{align*}

We conclude that the covariance matrix ${\rm Cov}(B_E)$ of $B_E$ has the following form:
all its diagonal entries are $= \frac{1}{2(n+2)} + O\Bigl(\frac{1}{n^2}\Bigr)$,
while the only non-zero non-diagonal entries are those giving the correlation between
marginals corresponding to two different diagonal entries of $T\in B_E$,
and these are $= - \frac{1}{4n(n+2)} + O\Bigl(\frac{1}{n^3}\Bigr)$.

As before, it follows that the volume-normalised unit ball $\overline{B_E}$ is in almost isotropic position.
This completes the proof of Theorem \ref{thm:almost-isotropic} in the orthogonal case too.
\qed

\section{Entrywise negative correlation property of $B_{\!\!{\cal M}_n({\mathbb R})}$ or $B_{\!\!{\cal M}_n({\mathbb C})}$} \label{sec:neg-cor-prop}

According to one of the main results in \cite{Radke-V-2016}, a necessary condition for the variance
conjecture to be true for the unit ball of any $p$-Schatten norm on ${\cal M}_n({\mathbb F})$ is 
that the corresponding density $f_{a,b,c}(x)\cdot e^{-\|x\|_p^p}\,dx$ appearing in Lemma \ref{lem:RMT-reduction}
and Proposition \ref{prop:var-reduction} satisfies a certain negative correlation property: more specifically, we need to have
\begin{equation}\label{eq:neg-cor-prop-of-Mp} \frac{M_p(x_i^2x_j^2)}{M_p(1)} 
= \frac{M_p\bigl(x_1^2x_2^2\bigr)}{M_p(1)} < \left(\frac{M_p\bigl(x_1^2\bigr)}{M_p(1)}\right)^2 = 
\frac{M_p\bigl(x_i^2\bigr)}{M_p(1)}\frac{M_p\bigl(x_j^2\bigr)}{M_p(1)}\end{equation}
for any $i\neq j$. This could be used to deduce similar inequalities for the original uniform densities on the unit balls 
of the $p$-Schatten norms which satisfy the conjecture: in \cite{Radke-V-2016} we showed that, if $p$ is large enough
(and, as a limiting case, if $p=\infty$ as well), then \eqref{eq:neg-cor-prop-of-Mp} holds true and, combined with the invariances of $K_{p,{\cal M}_n({\mathbb F})}$, implies that
\begin{equation*}
 \int_{\overline{K}_{p,{\cal M}_n({\mathbb F})}} |T_{i,j}|^2|T_{i,r}|^2\,dT 
 = \int_{\overline{K}_{p,{\cal M}_n({\mathbb F})}} |T_{j,i}|^2|T_{r,i}|^2\,dT 
  < \left(\int_{\overline{K}_{p,{\cal M}_n({\mathbb F})}} |T_{i,j}|^2\,dT\right) \left(\int_{\overline{K}_{p,{\cal M}_n({\mathbb F})}} |T_{i,r}|^2\,dT\right)
\end{equation*}
for all $i, j,r$, $j\neq r$.
However, it was unclear from our method whether a similar negative correlation property is true for the remaining pairs of entries, that is, when we consider the integrals $\int_{\overline{K}_{p,{\cal M}_n({\mathbb F})}}|T_{i,j}|^2|T_{l,r}|^2\,dT$ with $i\neq l$, $j\neq r$.

We can now check that it fails to be true and that we do not have negative correlation for the remaining pairs of entries of $T\sim {\rm Unif}\bigl(K_{\infty,{\cal M}_n({\mathbb F})}\bigr)$ 
when ${\mathbb F}$ is either ${\mathbb R}$ or ${\mathbb C}$ (of course it doesn't fail 
by much since the variance conjecture is correct in these cases). 
The key ingredients we will use to check this are the relevant tools in the Weingarten calculus 
coming from \cite{Collins-Matsumoto-Saad-2014} and the estimates we obtained in Section \ref{sec:main-proof}
(which also allow us to verify again the negative correlation property for pairs of entries
coming from the same row or the same column).

\medskip

\noindent\emph{Proof when ${\mathbb F} = {\mathbb C}$.}
To compute and compare the integrals
\begin{equation*} 
\frac{1}{{\rm vol}(K_\infty)}\int_{K_\infty} |T_{i,j}|^2|T_{l,r}|^2\,dT, \qquad \qquad  \left(\frac{1}{{\rm vol}(K_\infty)}\int_{K_\infty} |T_{1,1}|^2\,dT\right)^2,
\end{equation*}
we apply Theorem \ref{thm:CMS-C-lr-inv} with $k=2$ or $1$ respectively. 
Starting with the latter, we see that
\begin{align*}
\frac{1}{{\rm vol}(K_\infty)}\int_{K_\infty} |T_{1,1}|^2\,dT
&= \frac{2}{{\rm vol}(K_\infty)}\int_{K_\infty} \Re^2(T_{1,1})\,dT = 
\frac{2}{{\rm vol}(K_\infty)}\int_{K_\infty} \Im^2(T_{1,1})\,dT 
\\
&= \frac{1}{n^2} \cdot \frac{1}{{\rm vol}(K_\infty)}\int_{K_\infty}{\rm Tr}(TT^\ast) \,dT
\\
\intertext{as expected from the isotropicity of $\overline{K}_{\infty,{\cal M}_n({\mathbb C})}$,}
&= \frac{1}{n^2}\cdot \frac{N_\infty\bigl(\|x\|_2^2\bigr)}{N_\infty(1)} = \frac{1}{2n}.
\end{align*}

Moreover,
\begin{align*}
\frac{4}{{\rm vol}(K_\infty)}\int_{K_\infty} \Re^2(T_{i,j})\Re^2(T_{l,r})\,dT
& = \frac{4}{{\rm vol}(K_\infty)}\int_{K_\infty} \Im^2(T_{i,j})\Im^2(T_{l,r})\,dT
= \frac{4}{{\rm vol}(K_\infty)}\int_{K_\infty} \Re^2(T_{i,j})\Im^2(T_{l,r})\,dT
\\
= \frac{1}{{\rm vol}(K_\infty)}\int_{K_\infty} |T_{i,j}|^2|T_{l,r}|^2\,dT
& = \frac{1}{{\rm vol}(K_\infty)}\int_{K_\infty} T_{i,j}T_{l,r}\overline{T_{i,j}T_{l,r}}\,dT
\\
\intertext{and when $i\neq l$, $j\neq r$}
& = {\rm Wg}^U(e;n,n) \, \frac{1}{{\rm vol}(K_\infty)}\int_{K_\infty}\bigl({\rm Tr}(TT^\ast)\bigr)^2 \,dT
\\
& \qquad \qquad
+ {\rm Wg}^U((12);n,n) \, \frac{1}{{\rm vol}(K_\infty)}\int_{K_\infty}{\rm Tr}\bigl((TT^\ast)^2\bigr) \,dT
\\
& = \frac{n^2+1}{(n(n^2-1))^2} \frac{N_\infty\bigl(\|x\|_2^4\bigr)}{N_\infty(1)}
-\frac{2}{n(n^2-1)^2} \frac{N_\infty\bigl(\|x\|_4^4\bigr)}{N_\infty(1)}
\\
& = \frac{n^2+1}{(n(n^2-1))^2}\frac{n^4}{4n^2-1} - \frac{2}{n(n^2-1)^2} \frac{3n^3-n}{2(4n^2-1)}
\\
& = \frac{n^6 - 2n^4 + n^2}{n^2 (n^2-1)^2 (4n^2-1)} = \frac{1}{4n^2-1}.
\end{align*}
We thus see that
\begin{align*} 
\frac{1}{{\rm vol}(K_\infty)}\int_{K_\infty} |T_{i,j}|^2|T_{l,r}|^2\,dT &> \left(\frac{1}{{\rm vol}(K_\infty)}\int_{K_\infty} |T_{i,j}|^2\,dT\right) \left(\frac{1}{{\rm vol}(K_\infty)}\int_{K_\infty} |T_{i,r}|^2\,dT\right)  
\\
& > \left(1- O(1/n^2)\right) \frac{1}{{\rm vol}(K_\infty)}\int_{K_\infty} |T_{i,j}|^2|T_{l,r}|^2\,dT
\end{align*}
(the latter inequality being a necessary consequence of the variance conjecture holding true).

On the other hand,
\begin{align*}
\frac{1}{{\rm vol}(K_\infty)}\int_{K_\infty} |T_{i,j}|^2|T_{i,r}|^2\,dT
& = \frac{1}{{\rm vol}(K_\infty)}\int_{K_\infty} |T_{j,i}|^2|T_{r,i}|^2\,dT
\\
& = \left({\rm Wg}^U(e;n,n) + {\rm Wg}^U((12);n,n)\right)\cdot
\frac{1}{{\rm vol}(K_\infty)}\int_{K_\infty}\left(\bigl({\rm Tr}(TT^\ast)\bigr)^2 + {\rm Tr}\bigl((TT^\ast)^2\bigr)\right) \,dT
\\
& = \frac{1}{n^2(n+1)^2}\frac{2n^4 + 3n^3-n}{2(4n^2-1)} 
 = \frac{1}{2n(2n+1)} < \left(\frac{1}{{\rm vol}(K_\infty)}\int_{K_\infty} |T_{1,1}|^2\,dT\right)^2
\end{align*}
in accordance with the conclusions from \cite{Radke-V-2016}.
\qed

\bigskip
\bigskip

\noindent\emph{Proof when ${\mathbb F} = {\mathbb R}$.}
Applying Theorem \ref{thm:CMS-R-lr-inv} with $k=1$ or $2$, we can obtain:
\begin{equation*}
\frac{1}{{\rm vol}(K_\infty)}\int_{K_\infty} |T_{1,1}|^2\,dT
= \frac{1}{n^2} \cdot \frac{1}{{\rm vol}(K_\infty)}\int_{K_\infty}{\rm Tr}(TT^t) \,dT
= \frac{1}{n^2}\cdot \frac{N_\infty\bigl(\|x\|_2^2\bigr)}{N_\infty(1)} = \frac{1}{2n+1};
\end{equation*}
\begin{align*}
\frac{1}{{\rm vol}(K_\infty)}\int_{K_\infty} |T_{i,j}|^2|T_{l,r}|^2\,dT 
&= \sum_{\tau_1,\tau_2\in M_4} {\rm Wg}^O(\tau_1;n){\rm Wg}^O(\tau_2;n)\ \frac{1}{{\rm vol}(K_\infty)}\int_{K_\infty}{\rm Tr}^\prime_{\tau_1^{-1}\tau_2}(TT^t)\,dT
\\
& = \bigl(\bigl({\rm Wg}^O(e;n)\bigr)^2 + 2\bigl({\rm Wg}^O((23);n)\bigr)^2\bigr)\ \frac{1}{{\rm vol}(K_\infty)}\int_{K_\infty}\bigl({\rm Tr}(TT^t)\bigr)^2\,dT
\\
& \qquad \qquad + \bigl(4{\rm Wg}^O(e;n){\rm Wg}^O((23);n) + 2\bigl({\rm Wg}^O((23);n)\bigr)^2\bigr)
\ \frac{1}{{\rm vol}(K_\infty)}\int_{K_\infty}{\rm Tr}\bigl((TT^t)^2\bigr)\,dT
\\
& = \frac{n^2+2n+3}{(n(n-1)(n+2))^2}\frac{N_\infty\bigl(\|x\|_2^4\bigr)}{N_\infty(1)}
- \frac{4n+2}{(n(n-1)(n+2))^2}\frac{N_\infty\bigl(\|x\|_4^4\bigr)}{N_\infty(1)}
\\
& = \frac{n+1}{n(2n+1)(2n+3)}
\end{align*}
when $i\neq l$, $j\neq r$, while
\begin{align*}
\frac{1}{{\rm vol}(K_\infty)}\int_{K_\infty} |T_{i,j}|^2|T_{i,r}|^2\,dT
& = \frac{1}{{\rm vol}(K_\infty)}\int_{K_\infty} |T_{j,i}|^2|T_{r,i}|^2\,dT
\\
& = \sum_{\tau_1,\tau_2, \sigma_2\in M_4} {\rm Wg}^O(\tau_1;n){\rm Wg}^O(\sigma_2^{-1}\tau_2;n)\ \frac{1}{{\rm vol}(K_\infty)}\int_{K_\infty}{\rm Tr}^\prime_{\tau_1^{-1}\tau_2}(TT^t)\,dT
\\
& = \left(\sum_{\sigma_2\in M_4} {\rm Wg}^O(\sigma_2^{-1};n)\right)^2\frac{N_\infty\bigl(\|x\|_2^4\bigr)}{N_\infty(1)}
+ 2\,\left(\sum_{\sigma_2\in M_4} {\rm Wg}^O(\sigma_2^{-1};n)\right)^2\frac{N_\infty\bigl(\|x\|_4^4\bigr)}{N_\infty(1)}
\\
& = \frac{1}{(n(n+2))^2}\left(\frac{n^4 + n^3 + n}{(2n+1)(2n+3)} + \frac{3n^3+4n^2-n}{(2n+1)(2n+3)}\right)
\\
& = \frac{1}{(2n+1)(2n+3)}.
\end{align*}

These show that we have analogous conclusions as in the unitary case. \qed

\bigskip

\noindent {\bf Acknowledgement.} Research on this project was in part conducted while the author was in residence at the Mathematical Sciences Research Institute in Berkeley, California, during the fall semester of 2017. She gratefully acknowledges the support of the institute and of the National Science Foundation under Grant No. 1440140.

\bigskip

\bigskip

\medskip

\small

%\noindent \textsc{Beatrice-Helen Vritsiou:} Department of Mathematics,
%University of Michigan, 2074 East Hall, 530 Church Street, Ann Arbor, MI  48109-1043, USA
%
%\noindent {\it E-mail:} \texttt{vritsiou@umich.edu}

\noindent \textsc{Beatrice-Helen Vritsiou:} Department of Mathematical and Statistical Sciences,
University of Alberta, CAB 632, Edmonton, AB, Canada T6G 2G1

\noindent 
{\it E-mail:} \texttt{vritsiou@ualberta.ca}

\end{document}